\documentclass{amsart}
\usepackage{amssymb,amscd,verbatim, amsthm,amsmath,amsgen,xspace}
\usepackage{latexsym}
\usepackage{amsfonts}
\usepackage{color}
\usepackage{ulem}
\usepackage{csquotes}
\usepackage{pst-node}
\usepackage{tikz-cd} 
\usepackage[all]{xy}

\usepgflibrary{arrows}

\newcommand{\ka}{\mathfrak{k}}
\newcommand{\p}{\mathfrak{p}}
\newcommand{\g}{\mathfrak{g}}
\newcommand{\h}{\mathfrak{h}}
\newcommand{\C}{{\ensuremath{\mathbb{C}}}}

\newcommand\Span{\operatorname{span}}

\def\End{\mathop{\hbox {End}}\nolimits}

\DeclareMathOperator\Ker{Ker}

\newcommand{\pf}{\begin{proof}}
\newcommand{\epf}{\end{proof}}
\newcommand{\eq}{\begin{equation}}
\newcommand{\eeq}{\end{equation}}
\newcommand{\eqn}{\begin{equation*}}
\newcommand{\eeqn}{\end{equation*}}

\newcommand{\fra}{\mathfrak{a}}

\newcommand{\frg}{\mathfrak{g}}
\newcommand{\frh}{\mathfrak{h}}
\newcommand{\frk}{\mathfrak{k}}

\newcommand{\frq}{\mathfrak{q}}

\newcommand{\frt}{\mathfrak{t}}

\newcommand{\frsl}{\mathfrak{sl}}

\theoremstyle{plain}
\newtheorem{theorem}{Theorem}[section]
\newtheorem{cor}[theorem]{Corollary}
\newtheorem{prop}[theorem]{Proposition}
\newtheorem{lemma}[theorem]{Lemma}
\newtheorem{remark}[theorem]{Remark}

\newtheorem{definition}[theorem]{Definition}

\numberwithin{equation}{section}

\let\ssize\scriptstyle
\newif\ifFIRST\newdimen\MAXright\MAXright0pt
\def\sdynkin{\bgroup\eightpoint\dynkin}
\def\endsdynkin{\enddynkin\egroup}
\def\dynkin{\bgroup\FIRSTtrue\hskip.5em\setbox1\hbox{$\diagup$}%
	\setbox2\hbox{$\diagdown$}%
	\setbox0\hbox to2\wd1{\hrulefill}%
	\setbox3\hbox{$\bullet$}%
	\setbox4\hbox{$\times$}%
	\setbox7\hbox{$\circ$}
	\def\whiteroot##1{\ifFIRST\setbox5\hbox{$##1$}\ifdim\wd5>1.3em
		\hskip-.5em\hskip.5\wd5\fi\fi\FIRSTfalse
		\hskip-.25em\raise1.5\wd3\hbox to0pt{\hss\hskip.45em$
			\ssize##1$\hss}\copy7\hskip-.25em\setbox6\hbox{$##1$}
		\MAXright\wd6}
	\def\root##1{\ifFIRST\setbox5\hbox{$##1$}\ifdim\wd5>1.3em%
		\hskip-.5em\hskip.5\wd5\fi\fi\FIRSTfalse%
		\hskip-.25em\raise1.5\wd3\hbox to0pt{\hss\hskip.45em$%
			\ssize##1$\hss}\copy3\hskip-.25em\setbox6\hbox{$##1$}%
		\MAXright\wd6}%
	\def\whitedroot##1{\ifFIRST\setbox5\hbox{$##1$}\ifdim\wd5>1.3em
		\hskip-.5em\hskip.5\wd5\fi\fi\FIRSTfalse
		\hskip-.25em\lower1.8\wd3\hbox to0pt{\hss\hskip.45em$
			\ssize##1$\hss}\copy7\hskip-.25em\setbox6\hbox{$##1$}
		\MAXright\wd6}%
	\def\whiterroot##1{\hskip-.25em\copy7\hbox to0pt{\hskip.3em$\ssize##1$\hss}%
		\hskip-.25em\setbox6\hbox{\hskip.6em$##1##1$}%
		\MAXright\wd6}%
	\def\droot##1{\ifFIRST\setbox5\hbox{$##1$}\ifdim\wd5>1.3em%
		\hskip-.5em\hskip.5\wd5\fi\fi\FIRSTfalse%
		\hskip-.25em\lower1.8\wd3\hbox to0pt{\hss\hskip.45em$%
			\ssize##1$\hss}\copy3\hskip-.25em\setbox6\hbox{$##1$}%
		\MAXright\wd6}%
	\def\rroot##1{\hskip-.25em\copy3\hbox to0pt{\hskip.3em$\ssize##1$\hss}%
		\hskip-.25em\setbox6\hbox{\hskip.6em$##1##1$}%
		\MAXright\wd6}%
	\def\norroot##1{\hskip-.36em\copy4\hbox to0pt{\hskip.3em$\ssize##1$\hss}%
		\hskip-.48em\setbox6\hbox{\hskip.6em$##1##1$}%
		\MAXright\wd6}%
	\def\noroot##1{\ifFIRST\setbox5\hbox{$##1$}\ifdim\wd5>1.3em%
		\hskip-.5em\hskip.5\wd5\fi\fi\FIRSTfalse%
		\hskip-.36em\raise1.5\wd3\hbox to0pt{\hss\hskip.6em$%
			\ssize##1$\hss}\copy4\hskip-.38em\setbox6\hbox{$##1$}%
		\MAXright\wd6}%
	\def\nodroot##1{\ifFIRST\setbox5\hbox{$##1$}\ifdim\wd5>1.3em%
		\hskip-.5em\hskip.5\wd5\fi\fi\FIRSTfalse%
		\hskip-.36em\lower1.8\wd3\hbox to0pt{\hss\hskip.6em$%
			\ssize##1$\hss}\copy4\hskip-.38em\setbox6\hbox{$##1$}%
		\MAXright\wd6}%
	\def\nolink{\hskip\wd0}
	\def\link{\raise.22em\copy0}%
	\def\llink##1{\raise.32em\copy0\hskip-\wd0%
		\raise.12em\copy0\hskip-.5\wd0\hbox to0pt{\hss$##1$\hss}\hskip.5\wd0}%
	\def\lllink##1{\raise.22em\copy0\hskip-\wd0\raise.32em\copy0\hskip-\wd0%
		\raise.12em\copy0\hskip-.5\wd0\hbox to0pt{\hss$##1$\hss}\hskip.5\wd0}%
	\def\rootupright##1{\hbox to0pt{\raise.45em\copy1\hskip-.25em\raise1.3\ht1%
			\hbox{\copy3\hskip.3em$\ssize##1$}\hss}%
		\setbox6\hbox{\hskip.6em\copy1\copy1$##1##1$}%
		\ifdim\MAXright<\wd6\MAXright\wd6\fi}%
	\def\whiterootupright##1{\hbox to0pt{\raise.45em\copy1\hskip-.25em\raise1.3\ht1
			\hbox{\copy7\hskip.3em$\ssize##1$}\hss}
		\setbox6\hbox{\hskip.6em\copy1\copy1$##1##1$}
		\ifdim\MAXright<\wd6\MAXright\wd6\fi}
	\def\norootupright##1{\hbox to0pt{\raise.45em\copy1\hskip-.36em\raise1.3\ht1%
			\hbox{\copy4\hskip.3em$\ssize##1$}\hss}%
		\setbox6\hbox{\hskip.6em\copy1\copy1$##1##1$}%
		\ifdim\MAXright<\wd6\MAXright\wd6\fi}%
	\def\rootdownright##1{\hbox to0pt{\raise-.5em\copy2\hskip-.25em\raise-1.35\ht1%
			\hbox{\copy3\hskip.3em$\ssize##1$}\hss}\setbox6%
		\hbox{\hskip.6em\copy2\copy2$##1##1$}%
		\ifdim\MAXright<\wd6\MAXright\wd6\fi}%
	\def\whiterootdownright##1{\hbox to0pt{\raise-.5em\copy2\hskip-.25em\raise-1.35\ht1
			\hbox{\copy7\hskip.3em$\ssize##1$}\hss}\setbox6
		\hbox{\hskip.6em\copy2\copy2$##1##1$}
		\ifdim\MAXright<\wd6\MAXright\wd6\fi}
	\def\rootdown##1{\hbox to0pt{\hskip-.05em\vrule height.25em depth.65em%
			\hskip-.25em\raise-.95em\hbox{\copy3\hskip.3em$\ssize##1$}\hss}%
		\setbox6\hbox{$##1$}%
		\ifdim\MAXright<\wd6\MAXright\wd6\fi}%
	\def\whiterootdown##1{\hbox to0pt{\hskip-.05em\vrule height.25em depth.65em
			\hskip-.25em\raise-.95em\hbox{\copy7\hskip.3em$\ssize##1$}\hss}
		\setbox6\hbox{$##1$}
		\ifdim\MAXright<\wd6\MAXright\wd6\fi}
	\def\dots{\hskip.5em\cdots\hskip.5em}}%
\def\enddynkin{\ifdim\MAXright>1em\hskip.5\MAXright\else\hskip.5em\fi\egroup}%

\begin{document}

\date{\today}
\title[Construction of discrete series of $SO_{e}(4, 1)$ via Dirac induction]{Construction of discrete series representations of $SO_{e}(4, 1)$ via algebraic Dirac induction}
\author{Ana Prli\'{c}}
\address{Department of Mathematics, University of Zagreb, Bijeni\v cka 30, 10000 Zagreb, Croatia}
\email{anaprlic@math.hr}
\thanks{The author is supported by grant no. 4176 of the Croatian Science Foundation and by the QuantiXLie Centre of Excellence, a project
cofinanced by the Croatian Government and European Union through the
European Regional Development Fund - the Competitiveness and Cohesion
Operational Programme (Grant KK.01.1.1.01.0004). }
	
\begin{abstract}
The notion of algebraic Dirac induction was introduced by P. Pand\v{z}i\' c and D. Renard. This is a construction which gives representations with prescribed Dirac cohomology. They proved that all holomorphic discrete series representations can be constructed via Dirac induction. All discrete series representations of the group $SO_{e}(4,1)$ are nonholomorphic. In this paper we prove that they can also be constructed using algebraic Dirac induction. 
\end{abstract}

\maketitle

\section{Introduction}
The notion of Dirac cohomology was formulated by Vogan in 1997 \cite{V}. Let $G$ be a connected real reductive Lie group with Cartan involution $\Theta$ such that $K = G^{\Theta}$ is a maximal compact subgroup of $G$ and let $\g = \ka \oplus \p$ be the corresponding Cartan decomposition of the complexified Lie algebra $\g$ of $G$. The algebraic version of the Dirac operator $D$ is defined in \cite{V} as
$$
D = \sum_{i} b_i \otimes d_i \in U(\g) \otimes C(\p),
$$
where $U(\g)$ is the universal enveloping algebra of $\g$, $C(\p)$ is the Clifford algebra of $\p$ with respect to an invariant nondegenerate symmetric bilinear form $B$ on $\g_0$, $b_i$ is a basis of $\p$ and $d_i$ is the dual basis with respect to $B$. From the well known (and easy to prove) formula
$$
D(1 \otimes x) + (1 \otimes x)D = -2 x \otimes 1, \quad \text{for } x \in \p,
$$
it follows that understanding the $\p$ action on the $(\g, K)$--module $X$ comes down to understanding the action of the Dirac operator on $X \otimes S$ where $S$ is the spin module of $C(\p)$. The Dirac cohomology of a Harish-Chandra module $X$ is defined as
$$
H_{V}^{D}(X) = \text{Ker}D/ \text{Im}D \cap \text{Ker}D.
$$
It is a $\tilde{K}$--module, where $\tilde{K}$ is the spin double cover of $K$. In \cite{V} Vogan conjectured and in \cite{HP1} Huang and Pand\v{z}i\'c proved that Dirac cohomology, if nonzero, determines the infinitesimal character of the representation. Namely, the following theorem holds:
Let $\h=\frt\oplus\fra$ be a fundamental Cartan subalgebra of $\frg$. We view $\frt^*\subset\frh^*$ by extending functionals on $\frt$ by 0 over $\fra$. We fix compatible positive root systems $R_{\g}^{+}$ and $R_{\ka}^{+}$ for $(\g, \h)$ respectively $(\ka, \frt)$.
\begin{theorem}
\label{HPmain}
Let $X$ be a $(\frg,K)$-module with infinitesimal character $\Lambda\in\frh^*$. Assume that $H_V^{D}(X)$ contains the irreducible $K$--module $E_\gamma$ with highest weight $\gamma\in\frt^*$. Then 
$$
\gamma+\rho_\frk = w \Lambda \text{ for some } w \in W_{\g}.
$$
\end{theorem} 
One of the consequences of the above theorem is the fact that unitary $(\g, K)$--modules with nonzero Dirac cohomology are exactly those for which Dirac inequality becomes equality on some $\tilde{K}$--types of $X \otimes S$. Namely, Dirac inequality says that for a unitary $(\g, K)$--module $X$ and an irreducible representation $E_{\tau}$ of $\tilde{K}$ with the highest weight $\tau \in \h^{*}$ such that the multiplicity of the $\tilde{K}$--type $E_{\tau}$ in $X \otimes S$ is at least one, then 
$$
||\tau + \rho_{\ka}|| \geq || \Lambda ||,
$$
where $\Lambda \in \h^{*}$ denotes the infinitesimal character of $X$. Important examples of unitary modules with nonzero Dirac cohomology are discrete series representations, most of the more general unitary modules with strongly regular infinitesimal character, the so-called $A_{\frq}(\lambda)$--modules and many others. Classification of irreducible unitary $(\g, K)$--modules with nonzero Dirac cohomology is still an open problem.

In \cite{PR} the authors describe certain constructions which give representations with prescribed Dirac cohomology by tensoring the algebra $U(\g) \otimes C(\p)$ with the Dirac cohomology over a certain subalgebra that contains the Dirac operator. Since the functor of Dirac cohomology admits no adjoint, two alternative definitions were proposed in \cite{PR}, both of which coincide with the Vogan's definition for unitary and finite dimensional $(\g, K)$--modules. They are called Dirac cohomology and homology. The functor of Dirac cohomology is left exact and admits a right adjoint, while the functor of Dirac homology is right exact and admits a left adjoint. These adjoints are called Dirac induction functors. For example, all holomorphic discrete series can be constructed by tensoring the algebra $U(\g) \otimes C(\p)$ with the Dirac cohomology over the algebra generated by $\ka_{\Delta}$, the ideal $\mathcal{I} \subset (U(\g) \otimes C(\p))^{K}$ generated by the Dirac operator and $C(\p)^{K}$.

Let us review the definition and the main results about Dirac induction as defined in \cite{PR}. By composing the adjoint action of $\ka$ on $\p$ with the embedding of the Lie algebra $\mathfrak{so}(\p)$ into the Clifford algebra $C(\p)$, we get a Lie algebra map $\alpha: \ka \longrightarrow C(\p)$. Now we can embed the Lie algebra $\ka$ diagonally into $U(\g) \otimes C(\p)$, by
$$
X \mapsto X_{\Delta} = X \otimes 1 + 1 \otimes \alpha(X).
$$
We will denote the image of this embedding by $\ka_{\Delta}$ and $\mathcal{A} = U(\g) \otimes C(\p)$. Let $\mathcal{I}$ be the two-sided ideal in the algebra of $K$--invariants $\mathcal{A}^{K}$ generated by $D$ and let $\mathcal{B}$ be the $K$--invariant subalgebra of $\mathcal{A}$ with unit, generated by $\ka_{\Delta}$ and $\mathcal{I}$. 
\begin{definition}
Let $X$ be a $(\mathfrak{g}, K)$--module. Dirac cohomology of $X$ is defined as
$$
H^{D}(X) = \{ v \in X \otimes S \, | \,  av = 0, \forall a \in \mathcal{I}\},
$$
while the functor of the Dirac homology is defined as
$$
H_{D}(X) = X \otimes S / \mathcal{I}(X \otimes S).
$$
\end{definition}
If the module $X$ is finite-dimensional or unitary, then we have 
$$
H^{D}(X) = H_{D} (X) = H_{V}^{D}(X) = \text{Ker}(D) = \text{Ker}(D^2).
$$
The functor of the Dirac cohomology $H^{D} : \mathcal{M}(\mathcal{A}, \tilde{K}) \longrightarrow \mathcal{M}(\ka_{\Delta}, \tilde{K})$ has a left adjoint, the functor
$$
\text{Ind}_D : W \mapsto \mathcal{A} \otimes_{\mathcal{B}}W,
$$
while the functor of the Dirac homology $H^{D} : \mathcal{M}(\mathcal{A}, \tilde{K}) \longrightarrow \mathcal{M}(\ka_{\Delta}, \tilde{K})$ has a right adjoint, the functor
$$
\text{Ind}^D : W \mapsto \text{Hom}_{\mathcal{B}}(\mathcal{A}, W)_{\tilde{K}-\text{finite}}.
$$
The algebra $\mathcal{B}$ is the smallest reasonable choice for the algebra over which we tensor (that containes the Dirac operator). However, even for the simplest nontrivial example of a $(\g, K)$--module $(\mathfrak{sl}(2, \mathbb{C}), SO(2))$, the modules we get by this version of induction are not irreducible. Therefore, we would like to tensor over a bigger algebra $\mathcal{B}$. The second possible choice is the algebra $\mathcal{B} = U(\ka_{\Delta}) \mathcal{A}^{K}$ and the third one is the so-called ``intermediate" version where we tensor over the algebra $U(\ka_{\Delta})(C(\p)^{K} + \mathcal{I})$. In both of the last cases we can consider $H^{D}$ as a functor from the category of $(\mathcal{A}, \tilde{K})$--modules to the category of $(\mathcal{B}, \tilde{K})$--modules on which $\mathcal{I}$ acts by zero. Then the functor ${\text{Ind}}_D = \mathcal{A} \otimes_{\mathcal{B}} \cdot $ is left adjoint to the functor $H^{D}$, while the functor ${\text{Ind}}^D = \text{Hom}_{\mathcal{B}} ({\mathcal{A}}, \cdot)_{\tilde{K}-\text{finite}}$ is right adjoint to the functor $H_{D}$. It turns out that the algebra $C(\p)^{K}$ is very easy to describe unlike the algebra of all $K$--invariants in $\mathcal{A}$. The last one contains the algebra $U(\g)^{K}$ which is, in general, hard to describe. Since the algebra $\mathcal{B} = U(\ka_{\Delta})\mathcal{A}^{K}$ contains $U(\g)^K$ as $U(\g)^{K} \otimes 1$, it can be expected that this version of induction where we tensor over $\mathcal{B} = U(\ka_{\Delta})\mathcal{A}^{K}$ could be, in general, hard to describe. Nevertheless, the algebra $\mathcal{A}^{K}$ can be explicitly described for the groups $SU(2,1)$ and $SO_{e}(4,1)$. While all holomorphic discrete series representations can be constructed via intermediate version of Dirac induction, it turns out that for the construction of the nonholomorphic discrete series of the groups $SU(2,1)$ (see \cite{Pr1}) and $SO_{e}(4,1)$ we have to tensor over the algebra $U(\ka_{\Delta}) \mathcal{A}^{K}$. Although this does not look promising for generalization, since the algebra $\mathcal{A}^{K}$ is hard to describe, it seems to the author that for the construction of the more general examples of discrete series, it will be enough to tensor over the algebra generated by $U(\ka_{\Delta})$, $\mathcal{I}$, $C(\p)^{K}$ and $Z(\ka) \otimes 1$. The last algebra is exactly the algebra $U(\ka_{\Delta}) \mathcal{A}^{K}$ if $G$ is $SU(2,1)$ or $SO_{e}(4,1)$ (see \cite{Pr} and \cite{Pr2}).

The paper is organized as follows. In Section $2$ we give the $K$--types of the discrete series representation $A_{\mathfrak{b}}(\lambda)$ of $SO_{e}(4,1)$ using the results from \cite{M}. Further, we calculate the action of $\p$ on the highest weight vectors of the $K$--types of $A_{\mathfrak{b}}(\lambda)$ and we describe a basis of the $A_{\mathfrak{b}}(\lambda)$--module. In Section $3$, we reduce the obvious spanning set of the induced module to the spanning set which will later be shown to be a basis of the induced module. Finally, in Section $4$ we construct the discrete series of the group $SO_{e}(4,1)$ via algebraic Dirac induction.

In the future, we hope to generalize our results to all discrete series representations and to study the connection between Dirac induction, translation functors and cohomological induction. 

Finally, let us mention that the notion of the Dirac induction was recently introduced for graded affine Hecke algebras in \cite{COT}.  In that setting, the map $Ind_D$ is a Hecke algebra analog of the explicit realization of the Baum-Connes assembly map in the $K-theory$ of the reduced $C^{*}$--algebra of a real reductive group. This explicit realization uses Dirac operators.  Our version of induction does not appear to be directly related to the construction of \cite{COT}, but it would be very interesting to find some relations between the two constructions. We plan to study this question in near future.

\vspace{.2in}

\section{Discrete series representations of $SO_{e}(4,1)$}
The (real) Lie algebra of the Lie group $G = SO_{e}(4,1)$ is $\g_0 = \mathfrak{so}(4,1)$ .
We use the following basis of the complexification $\g$ of $\frg_0$ (see \cite{Pr2}):
\begin{gather*}
H_1 = i e_{12} - i e_{21},\quad H_2  = i e_{34} - i e_{43}, \\ 
E_1  = \frac{1}{2}(e_{13} - e_{24} - ie_{23} - ie_{14} - e_{31} + e_{42} + ie_{32} + ie_{41}), \\
E_2  = \frac{1}{2}(e_{13} + e_{24} - ie_{23} + ie_{14} - e_{31} - e_{42} + ie_{32} - ie_{41}), \\
F_1  = \frac{-1}{2}(e_{13} - e_{24} + ie_{23} + ie_{14} - e_{31} + e_{42} - ie_{32} - ie_{41}), \\
F_2  = \frac{-1}{2}(e_{13} + e_{24} + ie_{23} - ie_{14} - e_{31} - e_{42} - ie_{32} + ie_{41}), \\
E_3  = e_{15} - ie_{25} + e_{51} - ie_{52}, \\
E_4  = e_{35} - ie_{45} + e_{53} - ie_{54}, \\
F_3  = e_{15} + ie_{25} + e_{51} + ie_{52}, \\
F_4  = e_{35} + ie_{45} + e_{53} + ie_{54}.
\end{gather*}
where $e_{ij}$ denotes the matrix with the $ij$ entry equal to $1$ and all other entries equal to $0$. 
The commutation relations are given by the following table (see \cite{Pr2}): 
\begin{table}[ht]
\caption{commutator table}
\label{tab:commtable}
\begin{tabular}{|r|| r| r| r| r| r| r| r| r| r|}
  \hline
        & $H_2$ & $E_1$ & $E_2$ & $F_1$         & $F_2$       & $E_3$ & $E_4$  & $F_3$  & $F_4$ \\ \hline \hline
  $H_1$ & $0$   & $E_1$ & $E2$  & $-F_1$        & $-F_2$      & $E_3$ & $0$    & $-F_3$ & $0$   \\ \hline
  $H_2$ &       & $E_1$ & $-E2$ & $-F_1$        & $F_2$       & $0$   & $E_4$  & $0$    & $-F_4$   \\ \hline
  $E_1$ &       &       & $0$   & $H_1 + H_2$   & $0$         & $0$   & $0$    & $-E_4$ & $E_3$   \\ \hline
  $E_2$ &       &       &       & $0$           & $H_1 - H_2$ & $0$   & $E_3$  & $-F_4$ & $0$   \\ \hline
  $F_1$ &       &       &       &               & $0$         & $F_4$ & $-F_3$ & $0$    & $0$   \\ \hline
  $F_2$ &       &       &       &               &             & $E_4$ & $0$    & $0$    & $-F_3$   \\ \hline
  $E_3$ &       &       &       &               &             &       & $2E_1$ & $2H_1$ & $2E_2$   \\ \hline
  $E_4$ &       &       &       &               &             &       &        & $2F_2$ & $2H_2$   \\ \hline
  $F_3$ &       &       &       &               &             &       &        &        & $-2F_1$   \\ \hline
\end{tabular}
\end{table}

Let $\g = \ka \oplus \p$ be the Cartan decomposition of $\g$ corresponding to the Cartan involution $\theta (X) = - X^{\tau}$. Then
\[
\ka = \Span  \{ H_1, H_2, E_1, E_2, F_1, F_2 \} \cong \mathfrak{so}(4, \mathbb{C}),\text{ and } \p = \Span  \{ E_3, E_4, F_3, F_4 \}.
\]
We have $\ka = \ka_1 \oplus \ka_2$, where 
\[
\ka_1 = \Span\{ H_1 + H_2, E_1, F_1 \} \simeq \mathfrak{sl}(2, \mathbb{C}), \quad \ka_2 = \Span\{ H_1 - H_2, E_2, F_2\} \simeq \mathfrak{sl}(2, \mathbb{C})
\]
with $H_1 + H_2$, $E_1$ and $F_1$ (resp. $H_1 - H_2$, $E_2$ and $F_2$) corresponding to the standard basis of $\frsl(2,\mathbb{C})$. Algebras $\ka_1$ and $\ka_2$ mutually commute. In the standard positive root system, the compact positive roots are
$
\epsilon_1 \pm \epsilon_2,
$
and the noncompact positive roots are
$
\epsilon_1, \epsilon_2.
$
For a $\ka_1$--module $V$ and a $\ka_2$--module $W$, the external tensor product $V \boxtimes W$ is equal to the tensor product $V \otimes W$ as a vector space, and it is a $\ka$--module such that $\ka_1$ acts on $V$ and $\ka_2$ acts on $W$. We will denote by $V_m$ a finite dimensional $\mathfrak{sl}(2, \mathbb{C})$--module with the highest weight $m$. 

All discrete series representations are $A_{\mathfrak{b}}(\lambda)$--modules for some Borel subalgebra $\mathfrak{b}$ and some $\lambda \in \h^{*}$ which is is admissible for $\mathfrak{b}$. Here we take
$$
\mathfrak{b} = \h \oplus \g_{\epsilon_1 - \epsilon_2} \oplus \g_{\epsilon_1 + \epsilon_2} \oplus \g_{\epsilon_1} \oplus \g_{\epsilon_2}.
$$
From \cite[Corollary~7.]{M} it follows that the discrete series representation $A_{\mathfrak{b}}(\lambda)$ of $G$ has the following $K$--types
\begin{equation}\label{gama}
V_{\lambda_1 + \lambda_2 + 2 + l + k} \boxtimes V_{\lambda_1 - \lambda_2 - l + k}, \quad k \in \mathbb{N}_0, l \in \{ 0, 1, \cdots, \lambda_1 - \lambda_2\}. 
\end{equation}
Here $V_n \boxtimes V_m$ denotes the $\ka$--module with highest weight $(n, m)$, such that $H_1 + H_2$ acts on the highest weight vector $v_n \boxtimes v_m$ by $n$ and $H_1 - H_2$ acts on the highest weight vector $v_n \boxtimes v_m$ by $m$. All of the above $K$--types are of multiplicity one.  Let us denote the set (\ref{gama}) by $\Gamma(A_{\mathfrak{b}}(\lambda))$. The next step is to find the action of the elements of $\p$ on the highest weight vector $v_n \boxtimes v_m$ of each $K$--type $V_n \boxtimes V_m$ in $\Gamma(A_{\mathfrak{b}}(\lambda))$.

\begin{figure}[ht]
\caption{$K$--types}
\bigskip
\bigskip
\bigskip
\centering
\begin{tikzpicture}
\tikzstyle{every node}=[draw,shape=circle, fill, inner sep=0pt,minimum size=4pt]
\node[shape=circle,draw=black,label=left: \tiny $(m_0 \text{ , } n_0)$] at (0,0) {};
\node[shape=circle,draw=black,label=left: \tiny $(m_0 + 1 \text{ , } n_0 + 1)$] at (1.5,1.5) {};
\node[shape=circle,draw=black,label=left: \tiny $(m_0 + 2 \text{ , } n_0 + 2)$] at (3,3) {};

\filldraw [black] (4.5,4.5) circle (0.5pt);
\filldraw [black] (4.7,4.7) circle (0.5pt);
\filldraw [black] (4.9,4.9) circle (0.5pt);

\node[shape=circle,draw=black,label=left: \tiny $(m_0 + 1 \text{ , } n_0 - 1)$] at (1.5,-1.5) {};
\node[shape=circle,draw=black,label=left: \tiny $(m_0 + 2 \text{ , } n_0)$] at (3,0) {};
\node[shape=circle,draw=black,label=left: \tiny $(m_0 + 3 \text{ , } n_0 + 1)$] at (4.5,1.5) {};

\filldraw [black] (6, 3) circle (0.5pt);
\filldraw [black] (6.2, 3.2) circle (0.5pt);
\filldraw [black] (6.4, 3.4) circle (0.5pt);

\filldraw [black] (3, -3) circle (0.5pt);
\filldraw [black] (3.2, -3.2) circle (0.5pt);
\filldraw [black] (3.4, -3.4) circle (0.5pt);

\filldraw [black] (4.5, -1.5) circle (0.5pt);
\filldraw [black] (4.7, -1.7) circle (0.5pt);
\filldraw [black] (4.9, -1.9) circle (0.5pt);

\filldraw [black] (6, 0) circle (0.5pt);
\filldraw [black] (6.2, -0.2) circle (0.5pt);
\filldraw [black] (6.4, -0.4) circle (0.5pt);

\node[shape=circle,draw=black,label=left: \tiny $(m_0 + n_0 \text{ , } 0)$] at (5,-5) {};
\node[shape=circle,draw=black,label=left: \tiny $(m_0 + n_0 + 1 \text{ , } 1)$] at (6.5,-3.5) {};
\node[shape=circle,draw=black,label=left: \tiny $(m_0 + n_0 + 2 \text{ , } 2)$] at (8,-2) {};

\filldraw [black] (9.5,- 0.5) circle (0.5pt);
\filldraw [black] (9.7,- 0.3) circle (0.5pt);
\filldraw [black] (9.9, -0.1) circle (0.5pt);

\end{tikzpicture}
\end{figure}

In the above figure $m_0 = \lambda_1 + \lambda_2 + 2$ and $n_0 = \lambda_1 - \lambda_2$. The highest weight vector of the $K$--type $\p = \Span  \{ E_3, E_4, F_3, F_4 \} \simeq V_{1} \boxtimes V_{1}$ in $U(\g)$ is $E_3$. The weights of the vectors $E_4, F_3, F_4$ are $(1, -1), (-1, -1), (-1, 1)$, respectively. The action map $a: \p \otimes (V_n \boxtimes V_m) \longrightarrow \p \cdot V_n \boxtimes V_m$ is $\ka$--equivariant. Furthermore, we have
\begin{align*}
(V_1 \boxtimes V_1) \otimes (V_n \boxtimes V_m) & = (V_1 \otimes V_n) \boxtimes (V_1 \otimes V_m) \\
& = (V_{n + 1} \oplus V_{n - 1}) \boxtimes (V_{m + 1} \oplus V_{m-1}) \\
& = (V_{n + 1} \boxtimes V_{m + 1}) \oplus (V_{n + 1} \boxtimes V_{m-1}) \\
& \oplus (V_{n - 1} \boxtimes V_{m + 1}) \oplus (V_{n - 1} \boxtimes V_{m-1}).
\end{align*}
It follows that 
$$
E_{3} \cdot v_{n} \boxtimes v_{m} = \mu_{(n, m)} v_{n + 1} \boxtimes v_{m + 1},
$$
for some scalar $\mu_{(n, m)}$, where $v_{n + 1} \boxtimes v_{m + 1}$ is the highest weight vector of the $K$--type $V_{n + 1} \boxtimes V_{m + 1}$ of $A_{\mathfrak{b}}(\lambda)$. From now on, we will denote the $K$--type $V_n \boxtimes V_m$ by $(n, m)$.

\begin{lemma}\label{E3}
For each $K$--type $(n, m)$ of the discrete series $A_{\mathfrak{b}}(\lambda)$ we have $\mu_{(n, m)} \neq 0$, where $\mu_{(n, m)}$ is as above. Equivalently, $E_3 \cdot v_n \boxtimes v_m \neq 0$.
\end{lemma}
\begin{proof}
Let us fix some $K$--type $(n, m)$ of the discrete series $A_{\mathfrak{b}}(\lambda)$ with the highest weight vector $v_n \boxtimes v_m$.
Since $A_{\mathfrak{b}}(\lambda)$ is an irreducible $(\g, K)$--module, we have
$$
A_{\mathfrak{b}}(\lambda) = U(\g) \cdot v_n \boxtimes v_m.
$$
It follows from the Poincar\'{e}-Birkhoff-Witt theorem that
$$
A_{\mathfrak{b}}(\lambda) = U(\ka) \cdot \Span_{\mathbb{C}} \{ F_{4}^{a} F_{3}^{b} E_{4}^{c} E_{3}^{d} \, | \, a, b, c, d \in \mathbb{N}_{0} \} \cdot v_n \boxtimes v_m.
$$
Let us assume that $\mu_{n, m} = 0$, i.e., $E_3 \cdot v_n \boxtimes v_m = 0$. Then we have 
\begin{equation}\label{434}
A_{\mathfrak{b}}(\lambda) = U(\ka) \cdot \Span_{\mathbb{C}} \{ F_{4}^{a} F_{3}^{b} E_{4}^{c} \, | \, a, b, c \in \mathbb{N}_{0} \} \cdot v_n \boxtimes v_m.
\end{equation}
The weight of the vector $F_{4}^{a} F_{3}^{b} E_{4}^{c} \cdot v_n \boxtimes v_m$ is equal to $(n + c - b - a, m - c - b + a)$.
Since the $K$--type $(n + 1, m + 1)$ contributes to $A_{\mathfrak{b}}(\lambda)$, then it follows from \eqref{434} that 
\begin{align*}
n + 1 & = n + c - b  - a \\
m + 1 & = m - c - b  + a
\end{align*}
for some $a, b, c \in \mathbb{N}_{0}$. It follows that $b = -1$, a contradiction. Hence $\mu_{(n, m)} \neq 0$.
\end{proof}
Since the weights of the vectors $E_4$, $F_3$ and $F_4$ are $(1, -1)$, $(-1, -1)$ and $(-1, 1)$ respectively, it follows that 
\begin{align*}
E_4 \cdot v_n \boxtimes v_m & = \alpha_{(n , m)} F_2 \cdot v_{n + 1} \boxtimes v_{m + 1} + \beta_{(n, m)} v_{n + 1} \boxtimes v_{m - 1} \\
F_3 \cdot v_n \boxtimes v_m & = \gamma_{(n , m)} F_1 F_2 \cdot v_{n + 1} \boxtimes v_{m + 1} + \delta_{(n, m)} F_1 \cdot v_{n + 1} \boxtimes v_{m - 1} \\
& + \epsilon_{(n, m)} F_2 \cdot v_{n - 1} \boxtimes v_{m + 1} + \eta_{(n, m)} v_{n - 1} \boxtimes v_{m - 1} \\
F_4 \cdot v_n \boxtimes v_m & = \rho_{(n , m)} v_{n - 1} \boxtimes v_{m + 1} + \zeta_{(n, m)} F_1 \cdot v_{n + 1} \boxtimes v_{m + 1}, 
\end{align*}
where $v_{n} \boxtimes v_{m}, v_{n + 1} \boxtimes v_{m + 1}, v_{n + 1} \boxtimes v_{m - 1}, v_{n - 1} \boxtimes v_{m + 1}$ and $v_{n - 1} \boxtimes v_{m - 1}$ are the highest weight vector of the $K$--types $(n, m), (n + 1, m + 1), (n + 1, m - 1), (n - 1, m + 1)$ and $(n - 1, m - 1)$ in $A_{\mathfrak{b}}(\lambda)$. If some of these $K$--types do not contribute to $A_{\mathfrak{b}}(\lambda)$, then the corresponding scalar in the above equations equals $0$.  
\begin{figure}[ht]
\caption{$\p$--action}
\bigskip
\bigskip
\centering
\begin{tikzpicture}
\tikzstyle{every node}=[draw,shape=circle, inner sep=0pt,minimum size=4pt]
\node[shape=circle,draw=black, fill, label=below: \tiny $(n \text{ , } m)$, label=left: $E_3 \qquad \qquad \qquad$] at (0,0) {};
\node[shape=circle,draw=black, fill, label=above: \tiny $(n + 1 \text{ , } m + 1)$] at (1,1) {};
\draw [black,  -triangle 60] (0.2,0.2) -- (0.8,0.8);

\node[shape=circle,draw=black,fill, label=left: \tiny $(n \text{ , } m)$, label=left: $E_4 \qquad \qquad \qquad$] at (6,0) {};
\node[shape=circle,draw=black,label=above: \tiny $(n + 1 \text{ , } m + 1)$] at (7,1) {};
\node[shape=circle,draw=black, fill, label=below: \tiny $(n + 1 \text{ , } m - 1)$] at (7,-1) {};
\draw [black,  -triangle 60] (6.2,0.2) -- (6.8,0.8);
\draw [black,  -triangle 60] (6.2,-0.2) -- (6.8,- 0.8);

\node[shape=circle,draw=black,fill, label=below: \tiny $(n \text{ , } m)$, label=left: $F_3 \qquad \qquad \qquad$] at (0,-6) {};
\node[shape=circle,draw=black,label=above: \tiny $(n + 1 \text{ , } m + 1)$] at (1, -5) {};
\node[shape=circle,draw=black, label=below: \tiny $(n + 1 \text{ , } m - 1)$] at (1,-7) {};
\node[shape=circle,draw=black, fill, label=below: \tiny $(n - 1 \text{ , } m - 1)$] at (-1,-7) {};
\node[shape=circle,draw=black, label=above: \tiny $(n - 1 \text{ , } m + 1)$] at (-1,-5) {};

\draw [black,  -triangle 60] (0.2, - 5.8) -- (0.8, -5.2);
\draw [black,  -triangle 60] (0.2, - 6.2) -- (0.8, -6.8);
\draw [black,  -triangle 60] (-0.2, - 6.2) -- (-0.8, -6.8);
\draw [black,  -triangle 60] (-0.2, - 5.8) -- (-0.8, -5.2);

\node[shape=circle,draw=black,fill, label=below: \tiny $(n \text{ , } m)$, label=left: $F_4 \qquad \qquad \qquad$] at (6,-6) {};
\node[shape=circle,draw=black,label=above: \tiny $(n + 1 \text{ , } m + 1)$] at (7, -5) {};
\node[shape=circle,draw=black, fill, label=above: \tiny $(n - 1 \text{ , } m + 1)$] at (5,-5) {};
\draw [black,  -triangle 60] (6.2, - 5.8) -- (6.8, -5.2);
\draw [black,  -triangle 60] (5.8, -5.8) -- (5.2, -5.2);
                  
\end{tikzpicture}
\end{figure}
\newpage
The black dots in the above figure represent the highest weight vectors of the corresponding $K$--types while the white dots represent lower vectors.
\begin{lemma}\label{E4}
Let us fix some $K$--type $(n, m)$ in $A_{\mathfrak{b}}(\lambda)$. With notation as above, we have $\beta_{(n, m)} \neq 0$ if the $K$--type $(n + 1, m - 1)$ belongs to $A_{\mathfrak{b}}(\lambda)$. 
\end{lemma}
\begin{proof}
Let us assume that the $K$--type $(n + 1, m - 1)$ belongs to $A_{\mathfrak{b}}(\lambda)$ and $\beta_{(n, m)} = 0$. From that assumption it follows that $\alpha_{(n, m)} \neq 0$. Otherwise we would have  $E_4 \cdot v_n \boxtimes v_m = 0$ and then $E_3 \cdot v_n \boxtimes v_m = 0$, since $E_3 = [E_2, E_4]$, which contradicts the previous lemma. Furthermore we have 
\begin{align*}
E_4 \cdot v_n \boxtimes v_m & = [F_3, E_1] \cdot v_n \boxtimes v_m = - E_1 F_3 \cdot v_n \boxtimes v_m \\
& = - E_1 (\gamma_{(n , m)} F_1 F_2 \cdot v_{n + 1} \boxtimes v_{m + 1} + \delta_{(n, m)} F_1 \cdot v_{n + 1} \boxtimes v_{m - 1} \\
& + \epsilon_{(n, m)} F_2 \cdot v_{n - 1} \boxtimes v_{m + 1} + \eta_{(n, m)} v_{n - 1} \boxtimes v_{m - 1}) \\
& = - \gamma_{(n , m)} E_1 F_1 F_2 \cdot v_{n + 1} \boxtimes v_{m + 1} - \delta_{(n, m)} E_1 F_1 \cdot v_{n + 1} \boxtimes v_{m - 1},
\end{align*}
since $E_1 \cdot v_{n - 1} = 0$. We assumed that $\beta(n, m) = 0$. Then $E_4 \cdot v_n \boxtimes v_m  = \alpha_{(n , m)} F_2 \cdot v_{n + 1} \boxtimes v_{m + 1}$ from where it follows $\delta_{(n, m)} = 0$ and then 
$$
\p \cdot V_n \boxtimes V_m \subset (V_{n + 1} \boxtimes V_{m + 1}) \oplus (V_{n - 1} \boxtimes V_{m + 1}) \oplus (V_{n - 1} \boxtimes V_{m - 1}).
$$
Since the $K$--type $(n + 1, m + 1)$ is in $A_{\mathfrak{b}}(\lambda)$, we have   
\begin{equation}\label{E4n1m1}
E_4 \cdot v_{n + 1} \boxtimes v_{m + 1}  = \alpha_{(n + 1 , m + 1)} F_2 \cdot v_{n + 2} \boxtimes v_{m + 2} + \beta_{(n + 1, m + 1)} v_{n + 2} \boxtimes v_{m}. 
\end{equation}
If the $K$--type $(n - 1, m - 1)$ is in $A_{\mathfrak{b}}(\lambda)$, then we have  
$$
E_4 \cdot v_{n - 1} \boxtimes v_{m - 1}  = \alpha_{(n - 1 , m - 1)} F_2 \cdot v_{n} \boxtimes v_{m} + \beta_{(n - 1, m - 1)} v_{n} \boxtimes v_{m - 2}.
$$
Furthermore, we have
\begin{align*}
E_4 \cdot v_{n + 1} \boxtimes v_{m + 1} & = E_4 \left ( \frac{1}{\mu_{(n, m)}} E_3 \cdot v_n \boxtimes v_m \right ) \\
& =  \frac{1}{\mu_{(n, m)}} \left ( E_3 E_4 - 2 E_1 \right ) \cdot v_n \boxtimes v_m  \\
& =  \frac{\alpha_{(n , m)}}{\mu_{(n, m)}} \left ( E_3 F_2 \right )  \cdot v_{n + 1} \boxtimes v_{m + 1}  \\
& =  \frac{\alpha_{(n , m)}}{\mu_{(n, m)}} \left ( - E_4 + F_2 E_3  \right ) \cdot v_{n + 1} \boxtimes v_{m + 1}, 
\end{align*}
from where it follows
$$
\left ( 1 + \frac{\alpha_{(n , m)}}{\mu_{(n, m)}} \right ) E_4 \cdot v_{n + 1} \boxtimes v_{m + 1} = \frac{\alpha_{(n , m)}}{\mu_{(n, m)}} \mu_{(n + 1, m + 1)} F_2 \cdot v_{n + 2} \boxtimes v_{m + 2}.
$$
Since $\alpha_{(n , m)} \neq 0$, $\mu_{(n ,m)} \neq 0$, $\mu_{(n + 1, m + 1)} \neq 0$ and $F_2 \cdot v_{n + 2} \boxtimes v_{m + 2} \neq 0$, then from the equation above it follows that $1 + \frac{\alpha_{(n , m)}}{\mu_{(n, m)}} \neq 0$. From here and \eqref{E4n1m1} we get $\beta_{(n + 1, m + 1)} = 0$. Similar calculation shows that if the $K$--type $(n - 1, m - 1)$ is in $A_{\mathfrak{b}}(\lambda)$, then $\beta_{(n - 1, m - 1)} = 0$. Since $\beta_{(n + 1, m + 1)} = \beta_{(n - 1, m - 1)} = 0$, then 
$$
\delta_{(n + 1, m + 1)} = \delta_{(n - 1, m - 1)} = 0.
$$
From the previous conclusions we get 
$$
\Span_{\mathbb{C}} \{ F_{4}^{a} F_{3}^{b} E_{4}^{c} E_{3}^{d} \, | \, a, b, c, d \in \mathbb{N}_{0} \} \cdot V_n \boxtimes V_m \subset \bigcup_{\substack{(k, l) \in \Gamma(A_{\mathfrak{b}}(\lambda)) \\ k - l \leq n - m}} V_{k} \boxtimes V_l.
$$
This contradicts the assumption $(n + 1, m - 1) \in \Gamma(A_{\mathfrak{b}}(\lambda))$.
\end{proof}
\begin{theorem}\label{basis1}
The set 
\begin{align}\label{firstbasis}
\{ F_{1}^{a} F_{2}^{b} E_{3}^{k} E_{4}^{l} \cdot v_{\lambda_1 + \lambda_2  + 2} \boxtimes v_{\lambda_1 - \lambda_2} \, | \, & k \in \mathbb{N}_{0}, l \in \{ 0, 1, \cdots, \lambda_1 - \lambda_2 \}, \notag \\
& a \in \{ 0, 1, \cdots , \lambda_1 + \lambda_2  + 2 + k + l \}, \notag \\
& b \in \{ 0, 1, \cdots , \lambda_1 - \lambda_2 + k - l \} \}
\end{align}
is a basis of the discrete series representation $A_{\mathfrak{b}}(\lambda)$.
\end{theorem}
\begin{proof}
Let us introduce the following notation
\begin{align*}
S_{t} = \{ F_{1}^{a} F_{2}^{b} E_{3}^{k} E_{4}^{l} \cdot v_{\lambda_1 + \lambda_2  + 2} \boxtimes v_{\lambda_1 - \lambda_2} \, | \, & k \in \mathbb{N}_{0}, l \in \{ 0, 1, \cdots, \lambda_1 - \lambda_2 \}, k + l \leq t, \notag \\
& a \in \{ 0, 1, \cdots , \lambda_1 + \lambda_2  + 2 + k + l \}, \notag \\
& b \in \{ 0, 1, \cdots , \lambda_1 - \lambda_2 + k - l \} \},
\end{align*}
for $t \in \mathbb{N}_{0}$. From lemma \ref{E3} and lemma \ref{E4} it easily follows that the set $S_t$ is a spanning set of the vector space
\begin{equation}\label{Zt1}
Z_{t} = \bigoplus_{\substack{(\lambda_1 + \lambda_2 + 2 + r, y) \in \Gamma(A_{\mathfrak{b}}(\lambda)), \\ r \leq t}} V_{\lambda_1 + \lambda_2 + 2 + r} \boxtimes V_{y}.
\end{equation}
It is easy to see that
$$
Z_t =  \bigoplus_{r = 0}^{t} V_{m_0 + r} \boxtimes (V_{n_0 + r} \oplus V_{n_0 + r - 2} \oplus \cdots \oplus V_{|n_0 - r|}), 
$$
where $m_0 = \lambda_1 + \lambda_2 + 2$ and $n_0 = \lambda_1 - \lambda_2$. The dimension of that space is
\begin{align*}
\sum_{r = 0}^{t}& (m_0 + r + 1)(\text{dim}V_{n_0 + r} + \text{dim}V_{n_0 + r - 2} + \cdots + \text{dim}V_{|n_0 - r|}) \\
& = \sum_{r = 0}^{t}(m_0 + r + 1)(r + 1)(n_0 + 1) \\
& = \frac{1}{6} (n_0 + 1) (3 m_0 + 2 t + 3) (t^2 + 3t + 2),
\end{align*}
and the cardinality of the set $S_t$ is 
$$
\sum_{l = 0}^{n_0} \sum_{k = 0}^{t - l} (m_0 + k + l + 1)(n_0 + k - l + 1) = \frac{1}{6} (n_0 + 1) (3 m_0 + 2 t + 3) (t^2 + 3t + 2).
$$
From this it follows that $S_t$ is a basis of $Z_{t}$ and hence, the set (\ref{firstbasis}) is a basis of the discrete series $A_{\mathfrak{b}}(\lambda)$.
\end{proof}

\section{Dirac cohomology and the induced module}
The algebra $\mathcal{A}^{K}$ is generated by the Casimir elements for $\ka_1$ and $\ka_2$, by one generator from $C(\p)^{K}$, by the Dirac operator 
$$
D =  E_3 \otimes F_3 + E_4 \otimes F_4 + F_3 \otimes E_3 + F_4 \otimes E_4,
$$
and by the $\ka$--Dirac operator 
\begin{align}\label{k-Dirac}
D_{\ka} & = E_1 \otimes \alpha(2 F_1) + E_2 \otimes \alpha(2 F_2) + F_1 \otimes \alpha(2 E_1) + F_2 \otimes \alpha(2 E_2) \notag \\
& + (H_1 - H_2) \otimes \alpha(H_1 - H_2) + (H_1 + H_2) \otimes \alpha(H_1 + H_2).
\end{align}
For more details about the structure of the algebra $\mathcal{A}^{K}$ see \cite{Pr2}. The spin module is 
$$
S = \text{span} \{ 1, E_3, E_4, E_3 \wedge E_4 \}.
$$
The $C(\p)$--module $S$  can be viewed as a $\ka$--module (see \cite[Section~2.]{HP2}). We will denote by $V_{(n, m)}$ the irreducible finite dimensional $\ka$--module with highest weight $(n, m)$ such that $H_1$ acts by $n$ on the highest weight vector, while $H_2$ acts by $m$ on the highest weight vector. As a $\ka$--module $S$ is equal to
$$
S = V_{(\frac{1}{2},\frac{1}{2})} \oplus V_{(\frac{1}{2},- \frac{1}{2})},
$$
where 
\begin{align*}
V_{(\frac{1}{2},\frac{1}{2})} & \simeq V_{1} \boxtimes V_{0}  = \text{span} \{ E_3 \wedge E_4, 1 \} \\
V_{(\frac{1}{2}, - \frac{1}{2})} & \simeq V_{0} \boxtimes V_{1}  = \text{span} \{ E_3, E_4 \}.
\end{align*}
Since the lowest $K$--type of the discrete series representation $A_{\mathfrak{b}}(\lambda)$ is $$V_{\lambda_1 + \lambda_2 + 2} \boxtimes V_{\lambda_1 - \lambda_2} \simeq V_{(\lambda_1 + 1, \lambda_2 + 1)},$$ then by \cite[Corollary~7.4.5.]{HP2} it follows that the Dirac cohomology $H^{D}(A_{\mathfrak{b}}(\lambda))$ is a single $\tilde{K}$--type with the highest weight equal to $(\lambda_1 + 1, \lambda_2 + 1) - \rho_n$, where $\rho_n = \rho - \rho_{\ka} = \frac{1}{2}(1, 1)$ is the half sum of the noncompact positive roots. Therefore, the Dirac cohomology of 
 $A_{\mathfrak{b}}(\lambda)$ is the $\tilde{K}$--type $V_{(\lambda_1 + \frac{1}{2}, \lambda_2 + \frac{1}{2})}$, from where it follows that the Dirac cohomology is contained in the tensor product of the lowest $K$--type of $A_{\mathfrak{b}}(\lambda)$ and the $K$--type $V_{(\frac{1}{2},\frac{1}{2})} \simeq V_{1} \boxtimes V_{0} = \text{span} \{ E_3 \wedge E_4, 1 \}$ in $S$. Now, using the fact that the highest weight vector of $H^{D}(A_{\mathfrak{b}}(\lambda))$ is annihilated by the elements $(E_{1})_{\Delta}$ and $(E_{2})_{\Delta}$, it is easy to check that $H^{D}(A_{\mathfrak{b}}(\lambda))$ is the $\tilde{K}$--type with the highest weight vector equal to
$$
w_{(\lambda_1 + \frac{1}{2}, \lambda_2 + \frac{1}{2})} = (v_{\lambda_1 + \lambda_2} \boxtimes v_{\lambda_1 - \lambda_2}) \otimes E_{3} \wedge E_{4} + 2 (\lambda_1 + \lambda_2 + 2) (v_{\lambda_1 + \lambda_2 + 2} \boxtimes v_{\lambda_1 - \lambda_2}) \otimes 1.
$$
The Dirac cohomology is 
$$
H^{D}(A_{\mathfrak{b}}(\lambda)) = \text{span} \{ w_{(\lambda_1 + \frac{1}{2} - s, \lambda_2 + \frac{1}{2} - t)} \, | \, s \in \{ 0, 1, \cdots, \lambda_1 + \lambda_2 + 1 \}, t \in \{ 0, 1, \cdots, \lambda_1 - \lambda_2\}\}, 
$$
where 
\begin{align*}
w_{(\lambda_1 + \frac{1}{2} - s,  \lambda_2 + \frac{1}{2} - t)} & = (F_{1})_{\Delta}^{s} (F_{2})_{\Delta}^{t}  w_{(\lambda_1 + \frac{1}{2}, \lambda_2 + \frac{1}{2})} \\
& = (v_{\lambda_1 + \lambda_2 - 2s} \boxtimes v_{\lambda_1 - \lambda_2 - 2t}) \otimes E_{3} \wedge E_{4} \\
& + 2 (\lambda_1 + \lambda_2 + 2 - s) (v_{\lambda_1 + \lambda_2 + 2 - 2s} \boxtimes v_{\lambda_1 - \lambda_2 - 2t}) \otimes 1.
\end{align*}
Let $x_s$ denote $\lambda_1 + \frac{1}{2} - s$ and let $y_t$ denote $\lambda_2 + \frac{1}{2} - t$. We describe a spanning set of $(\mathcal{A}, \tilde{K})$--module $\text{Ind}_{D}(W) = \mathcal{A} \otimes_{\mathcal{B}} W $, where $W = H^{D}(A_{\mathfrak{b}}(\lambda))$. Let us denote
$$
Z = \text{span}_{\mathbb{C}} \{ ab \otimes w - a \otimes bw \, | \, a \in \mathcal{A}, b \in \mathcal{B}, w \in W \}.
$$
We are going to reduce the obvious spanning set of $\mathcal{A} \otimes_{\mathcal{B}}W$ given by
$$
\{ a \otimes w + Z \, | \, a \in \mathcal{A}, w \in W \}.
$$
As in \cite{Pr1}, it is easy to show that we can ``remove" $U(\ka)$. For $x \in \mathfrak{k}$, $y \in C(\mathfrak{p})$ and $w \in W$ we have
\begin{align*}
(x \otimes y) \otimes w & =\left((1 \otimes y) \cdot (x_{\Delta} - 1 \otimes \alpha(x)) \right) \otimes w = \\
& = ((1 \otimes y) \cdot x_{\Delta}) \otimes w - (1 \otimes y \cdot \alpha(x)) \otimes w \\
& = \underbrace{((1 \otimes y) \cdot x_{\Delta}) \otimes w - (1 \otimes y) \otimes x_{\Delta}w}_{\in Z} + (1 \otimes y) \otimes x_{\Delta}w \\
& - (1 \otimes y \cdot \alpha(x)) \otimes w \\
& \in \left(1 \otimes C(\mathfrak{p}) \right) \otimes W + Z.
\end{align*}
From here it follows:
\begin{equation}\label{elimk}
\left(U(\mathfrak{k}) \otimes C(\mathfrak{p})\right) \otimes W \subset \left( 1 \otimes C(\mathfrak{p}) \right) \otimes W + Z.
\end{equation}
Now, we are going to describe the structure of the algebra $U(\g)$ using the results from \cite{Pr2}. The next proposition and its proof are similar to \cite[Proposition~4.2.]{Pr1} and its proof, but we recall it for the reader convenience.
\begin{prop}\label{ug}
We have
\begin{align}\label{Ug}
U(\g) = \Span_{\C} \{ x(E_3 F_3 + E_4 F_4)^{t} y \, | \, & x \in \{(\text{ad}F_1)^a (\text{ad}F_2)^b (E_{3}^{k}) \, | \, k \in \mathbb{N}_{0}, a, b \in \{0, \cdots, k \} ,  \\ 
& t \in \mathbb{N}_{0}, y \in U(\ka) \}. \notag 
\end{align}
\end{prop}
\begin{proof}
Let $T$ be the right side of \eqref{Ug}. We will show by induction that $U_{l}(\g) \subset T$ for all $l \in \mathbb{N}_{0}$. The claim is obvious for $l = 0$ and $l = 1$. Let us assume that the claim is true for $l \in \mathbb{N}_{0}$ less than or equal to some fixed $t \geq 1$. By the Poincare-Birkhoff-Witt's theorem, it is enough to prove that elements of the form $E_{3}^{p} E_{4}^{q} F_{3}^{r} F_{4}^{s}$, $p,q,r,s \in \mathbb{N}_{0}, p+q+r+s = t+1$ are in $T$. From \cite[p.~278]{Pr2} it follows that
$$
S^{n}(\mathfrak{p}) = V_{(n, 0)} \oplus (E_3 F_3 + E_4 F_4) S^{n-2}(\mathfrak{p}),
$$
for all $n \in \mathbb{N}, n \geq 2$, where $V_{(n, 0)}$ is the irreducible $\mathfrak{k}$--module with the highest weight vector $E_{3}^{n}$ and the highest weight $(n, 0)$ (more precisely $H_1$ acts by $n$ and $H_2$ acts by zero on $E_{3}^{n}$). Now we have
\begin{align*}
E_{3}^{p} E_{4}^{q} F_{3}^{r} F_{4}^{s}  \in & \text{ (span}_{\C}\{ (\text{ad}F_1)^a (\text{ad}F_2)^b (E_{3}^{t + 1}) \, | \, k \in \mathbb{N}_{0}, a, b \in \ \{0, \cdots, t + 1 \} \} \} \\
\oplus  & \text{ span}_{\C} \{ E_{3}^{m+1} E_{4}^{n} F_{3}^{r+1} F_{4}^{s} + E_{3}^{m} E_{4}^{n+1} F_{3}^{r} F_{4}^{s+1} \, | \, m, n, r, s \in \mathbb{N}_{0}, \\
& m + n + r + s = t-1 \}\text{)} \oplus U_{t}(\g).
\end{align*}
Therefore, it is enough to show that
$
E_{3}^{m+1} E_{4}^{n} F_{3}^{r+1} F_{4}^{s} + E_{3}^{m} E_{4}^{n+1} F_{3}^{r} F_{4}^{s+1} \in T$ for $m, n, r, s \in \mathbb{N}_{0}$, $m + n + r + s = t-1$.
By induction on $r$ and $s$ it can be easily seen that
\begin{align*}
E_{3}^{m+1} E_{4}^{n} F_{3}^{r+1} F_{4}^{s} + E_{3}^{m} E_{4}^{n+1} F_{3}^{r} F_{4}^{s+1}  & \in E_{3}^{m} E_{4}^{n} F_{3}^{r} F_{4}^{s}(E_3 F_3 + E_4 F_4) + U_{t}(\g) \\
& \subset U_{t-1}(\g)(E_3 F_3 + E_4 F_4) + U_{t}(\g),
\end{align*}
and the proposition follows.
\end{proof}
\bigskip
Furthermore, we have
\begin{equation}\label{connection}
(E_3 F_3 + E_4 F_4) \otimes 1 = - \frac{1}{2}D^2 + D_{\ka} + (H_1 + H_2). 
\end{equation}
which implies that
$$
((E_3 F_3 + E_4 F_4) \otimes 1) \otimes w \in \left( U(\mathfrak{k}) \otimes C(\mathfrak{p}) \right) \otimes W + Z.
$$
By \eqref{elimk}, proposition \ref{ug} and \eqref{connection} we see that the vector space $\mathcal{A} \otimes_{\mathcal{B}} W$ is spanned by the elements of the form $(x \otimes y) \otimes w$ where

\begin{equation}\label{ug_reduced}
x \in \{ (\text{ad}F_1)^{a} (\text{ad}F_2)^{b} E_{3}^{k} \,| \, k \in \mathbb{N}_{0}, a, b \in \{ 0, 1, \cdots, k\} \}, \, y \in C(\mathfrak{p}), \, w \in W.
\end{equation}
The next step is to reduce a part of the algebra $C(\p)$.
\begin{lemma}\label{E3F4}
If $T \in \End_{\mathbb{C}}(S) \cong C(\mathfrak{p})$ is any linear operator such that
$$
T(1) = T(E_3 \wedge E_4) = 0,
$$
then $T = T p^{'}$, where $p^{'}$ is the projection on $\Span_{\mathbb{C}} \{ E_3, E_4 \}$.
\end{lemma}

\proof
For $s \in \text{span}_{\mathbb{C}} \{ E_3, E_4 \}$ we have
$p'(s) = s$ and then $Tp'(s) = T(s)$. Furthermore, $T(1) = T(E_3 \wedge E_4) = 0 = T p'(1) = T p'(E_3 \wedge E_4)$. Since $$S =  \text{span}_{\mathbb{C}}\{1, E_3, E_4, E_3 \wedge E_4 \}$$ it follows that $T = T p'$.
\qed

\bigskip

In the spin module $S$, the following identities hold
\begin{align*}
E_3 F_4 \cdot 1 & = 0, \quad & E_3 F_4 \cdot E_3 \wedge E_4 & = E_3 \cdot (2 E_3) = 0, \\
E_4 F_3 \cdot 1 & = 0, \quad & E_4 F_3 \cdot E_3 \wedge E_4 & = E_4 \cdot (-2 E_4) = 0, \\
(E_3 F_3 - E_4 F_4) \cdot 1 & = 0, \quad & (E_3 F_3 - E_4 F_4) \cdot E_3 \wedge E_4 & = E_3 \cdot (-2 E_4) - E_4 \cdot (2 E_3) = 0.
\end{align*}
By lemma \ref{E3F4}, we have $E_3 F_4 p' = E_3 F_4$, $E_4 F_3 p' = E_4 F_3$ and $(E_3 F_3 - E_4 F_4) p' = (E_3 F_3 - E_4 F_4)$. Since $(1 \otimes p')w = 0$ for $w \in W$, we get that for any $a \in \mathcal{A}$ and \newline $s \in \text{span}_{\mathbb{C}}\{ E_3 F_4, E_4 F_3, (E_3 F_3 - E_4 F_4) \}$ 
\begin{align*}
\left( a(1 \otimes s) \right) \otimes w & = (a(1 \otimes sp')) \otimes w  = (a(1 \otimes s)(1 \otimes p')) \otimes w \\
& = (a (1 \otimes s)(1 \otimes p')) \otimes w - (a(1 \otimes s)) \otimes \underbrace{(1 \otimes p')w}_{0} \in Z.
\end{align*}
It follows that
\begin{align}\label{EF}
& (U(\mathfrak{g}) \otimes C(\mathfrak{p})E_3 F_4) \otimes W \subset Z \notag \\ 
& (U(\mathfrak{g}) \otimes C(\mathfrak{p})E_4 F_3) \otimes W \subset Z \notag \\
& (U(\mathfrak{g}) \otimes C(\mathfrak{p})(E_3 F_3 - E_4 F_4)) \otimes W \subset Z.
\end{align}
From \eqref{elimk}, \eqref{ug_reduced}, \eqref{EF} and
$$
(\text{ad}F_{i})^{a} E_{3}^{k} \in \Span \{ F_{i}^{m} E_{3}^{k} F_{i}^{n} \, | \, m + n = a\}, \quad i = 1, 2
$$
it follows that the vector space $\mathcal{A} \otimes_{\mathcal{B}} W$ is spanned by the set
\begin{align}\label{cp_reduced}
\{ (x \otimes y) \otimes w \, | \,  & x \in \{ F_{1}^{a}F_{2}^{b} E_{3}^{k} \, | \, k \in \mathbb{N}_0, a, b \in \{ 0, \cdots, k \} \}, \\
& y \in \{1,  E_3, E_4, F_3, F_4, E_3 E_4, E_3 F_3, F_3 F_4 \}, w \in W \}. \notag
\end{align}
\bigskip
Furthermore, one can easily check that
\begin{align}\label{anticom1}
D \cdot (1 \otimes E_3 F_4) - (1 \otimes E_3 F_4) \cdot D & = -2(E_3 \otimes F_4 - F_4 \otimes E_3) ; \\
D \cdot (1 \otimes E_4 F_3) - (1 \otimes E_4 F_3) \cdot D & = -2(E_4 \otimes F_3 - F_3 \otimes E_4) \notag ; \\
D \cdot (1 \otimes (E_3 F_3 - E_4 F_4)) - (1 \otimes (E_3 F_3 - E_4 F_4)) \cdot D & = -2 D + 4 (E_4 \otimes F_4 + F_3 \otimes E_3) \notag ; \\
D - (E_4 \otimes F_4 + F_3 \otimes E_3) & = (E_3 \otimes F_3 + F_4 \otimes E_4). \notag
\end{align} 
From \eqref{anticom1}, \eqref{EF} and $D \otimes W \subset Z$ we get
\begin{align}\label{anticom}
(E_3 \otimes F_4 - F_4 \otimes E_3) \otimes w \in Z  \\ 
(E_4 \otimes F_3 - F_3 \otimes E_4) \otimes w \in Z \notag \\
(E_4 \otimes F_4 + F_3 \otimes E_3) \otimes w \in Z \notag \\
(E_3 \otimes F_3 + F_4 \otimes E_4) \otimes w \in Z. \notag
\end{align}

\begin{prop}\label{F4}
For $s \in \{0, 1, 2, \cdots, \lambda_1 + \lambda_2 \}$ and $t \in \{0, 1, 2, \cdots, \lambda_1 - \lambda_2 \}$ we have
$$(1 \otimes F_4) \otimes (\lambda_1 + \lambda_2 + 1 - s) w_{(x_s, y_t)} \in (1 \otimes E_3) \otimes (F_1)_{\Delta}w_{(x_s, y_t)} + Z.$$
\end{prop}

\proof
For $w \in W$ we have
\begin{equation}\label{1o}
(F_1)_{\Delta}(1 \otimes E_3) \otimes w 
 \in (1 \otimes E_3) \otimes (F_1)_{\Delta}w + (1 \otimes F_4) \otimes w + Z.
\end{equation}
Furthermore, using \eqref{EF} we get
\begin{equation}\label{2o}
(F_1)_{\Delta}(1 \otimes E_3) \otimes w  
\in (F_1 \otimes E_3) \otimes w + Z.
\end{equation}
From (\ref{anticom}), and $D \otimes W \subset Z$ we get
\begin{align}\label{3o}
(F_1 \otimes E_3) \otimes w & = - \frac{1}{2}([F_3, F_4] \otimes E_3) \otimes w \\
& = - \frac{1}{2}[(F_3 F_4 \otimes E_3) \otimes w - (F_4 F_3 \otimes E_3) \otimes w] \notag \\
& \in - \frac{1}{2}[(F_3 E_3 \otimes F_4) \otimes w + (F_4 E_4 \otimes F_4) \otimes w] + Z \notag \\
& = -\frac{1}{2}[(F_3 E_3 + F_4 E_4) \otimes F_4 \otimes w] + Z. \notag
\end{align}
From \eqref{connection}, we get 
\begin{equation}\label{conn}
(F_3 E_3 + F_4 E_4) \otimes 1 = - \frac{1}{2}D^2 + D_{\ka} - (H_1 + H_2). 
\end{equation}
A strightforward calculation shows that $D_{\ka}$ acts on $W$ by $(-\lambda_1 - \lambda_2 - 4)$. The proof follows from \eqref{1o}, \eqref{2o}, \eqref{3o}, \eqref{conn} and from
\begin{align*}
((H_1 + H_2) \otimes F_4) \otimes w_{(x_s, y_t)} 
& = (1 \otimes F_4) ((H_1 + H_2)_{\Delta} \otimes w_{(x_s, y_t)} \\
& - 1 \otimes \alpha(H_1 + H_2) \otimes w_{(x_s, y_t)}) \\
& = (\lambda_1 + \lambda_2 - 2s)(1 \otimes F_4) \otimes w_{(x_s, y_t)} + Z.
\end{align*}
\qed

\begin{cor}\label{F}
For $s \in \{ 1,2, \cdots, \lambda_1 + \lambda_2\}$ and $t \in \{0, 1, 2, \cdots, \lambda_1 - \lambda_2 \}$ we have
\begin{align*}
&(1 \otimes E_4 F_4)  \otimes (\lambda_1 + \lambda_2 + 1 - s) w_{(x_s, y_t)} 
 \in (1 \otimes E_4 E_3) \otimes (F_1)_{\Delta}w_{(x_s, y_t)} + Z, \\
& (1 \otimes F_3 F_4)  \otimes (\lambda_1 + \lambda_2 + 1 - s) w_{(x_s, y_t)} 
 \in (1 \otimes F_3 E_3) \otimes (F_1)_{\Delta}w_{(x_s, y_t)} + Z, \\
& (1 \otimes F_3)  \otimes (\lambda_1 + \lambda_2 + 1 - s) w_{(x_S, y_t)} 
 \in -(1 \otimes E_4) \otimes (F_1)_{\Delta}w_{(x_s, y_t)} + Z.
\end{align*}
\end{cor}

\proof
The first two claims follow from the previous lemma. Furthermore, we have
\begin{align*}
& (1 \otimes E_4 F_3 F_4)  \otimes (\lambda_1 + \lambda_2 + 1 - s) w_{(x_s, y_t)} \\
& \in (1 \otimes E_4 F_3 E_3) \otimes (F_1)_{\Delta}w_{(x_s, y_t)} + Z \\
\end{align*}
The proof follows from $E_4 F_3 E_3 = - 2 E_4 + E_3 E_4 F_3$, $E_4 F_3 F_4 = 2 F_3 + F_4 E_4 F_3$ and $(U(\mathfrak{g}) \otimes C(\mathfrak{p})E_4 F_3) \otimes W \subset Z$.
\qed

\bigskip

From Proposition \ref{F4}, Corollary \ref{F} and from the previous conclusions it follows that the vector space $\mathcal{A} \otimes_{\mathcal{B}} W$ is spanned by the set
\begin{align*}
\{ & (x \otimes E_3) \otimes w_{(x_s, y_t)}, (x \otimes E_4) \otimes w_{(x_s, y_t)}, \\
& (x \otimes 1) \otimes w_{(x_s, y_t)}, (x \otimes E_3 E_4) \otimes w_{(x_s, y_t)}, \\
& (x \otimes F_3) \otimes w_{(x_{\lambda_1 + \lambda_2 + 1}, y_t)}, (x \otimes F_4) \otimes w_{(x_{\lambda_1 + \lambda_2 + 1}, y_t)}, \\
& (x \otimes E_3 F_3) \otimes w_{(x_{\lambda_1 + \lambda_2 + 1}, y_t)}, (x \otimes F_3 F_4) \otimes w_{(x_{\lambda_1 + \lambda_2 + 1}, y_t)}, \\
& x \in \{ F_{1}^{a} F_{2}^{b} E_{3}^{k} \, | \, k \in \mathbb{N}_{0}, a, b \in \{ 0, 1, \cdots, k \} \}, s \in \{ 0, 1, \cdots, \lambda_1 + \lambda_2 + 1 \}, t \in \{ 0, 1, \cdots, \lambda_1 - \lambda_2 \} \}.
\end{align*}

\begin{prop}\label{l1}
For $s \in \{0, 1, \cdots, \lambda_1 + \lambda_2 \}$, $t \in \{0, 1, \cdots, \lambda_1 - \lambda_2 \}$ and for any nonnegative integer $k$ we have
\begin{align*}
(F_1 E_{3}^{k} \otimes E_3) \otimes w_{(x_s, y_t)} \in \Span_{\mathbb{C}}\{ (E_{3}^{k - 2} \otimes E_3) \otimes (E_2)_{\Delta} w_{(x_s, y_t)}, \\
(E_{3}^{k} \otimes E_3) \otimes (F_1)_{\Delta} w_{(x_s, y_t)} \} + Z,
\end{align*}
where $E_{3}^{-1} = E_{3}^{-2} = 0$.
\end{prop}

\proof
One can easily show by induction that for each $k \in \mathbb{N}$
\begin{equation}\label{j1}
F_1 E_{3}^{k} = -k(k-1) E_{3}^{k - 2} E_2 + k E_{3}^{k - 1} F_4 + E_{3}^{k} F_1.
\end{equation}
Furthermore, using $\alpha(E_2) = - \frac{1}{2} E_3 F_4$ and $\alpha(F_1) = \frac{1}{2} F_3 F_4$ and previous conclusions we get
\begin{equation}\label{j2} 
(E_2 \otimes E_3) \otimes w_{(x_s, y_t)}  \in (1 \otimes E_3) \otimes (E_{2})_{\Delta} w_{(x_s, y_t)} + Z, 
\end{equation}
and
\begin{equation}\label{j3}
(F_1 \otimes E_3) \otimes w_{(x_s, y_t)} \in \frac{\lambda_1 + \lambda_2 + 2 - s}{\lambda_1 + \lambda_2 + 1 - s}(1 \otimes E_3) \otimes (F_{1})_{\Delta} w_{(x_s, y_t)} + Z.
\end{equation}
From (\ref{anticom}) and Proposition (\ref{F4}) we get
\begin{align}\label{j4}
& (F_4 \otimes E_3) \otimes w_{(x_s, y_t)} = (E_3 \otimes F_4) \otimes w_{(x_s, y_t)} \\
& \in \frac{1}{\lambda_1 + \lambda_2 + 1 - s}(E_3 \otimes E_3) \otimes (F_{1})_{\Delta} w_{(x_s, y_t)} + Z. \notag 
\end{align}
The proof follows from \eqref{j1}, \eqref{j2}, \eqref{j3} and (\ref{j4}).
\qed

\begin{prop}\label{l2}
For any $t \in \{0, 1, \cdots, \lambda_1 - \lambda_2 \}$ and for any nonnegative integer $k$ we have
\begin{align*}
 & (F_1 E_{3}^{k} \otimes E_3) \otimes w_{(x_{\lambda_1 + \lambda_2 + 1}, y_t)} \\
 & \in \Span_{\mathbb{C}}\{ (E_{3}^{k - 2} \otimes E_3) \otimes (E_2)_{\Delta} w_{(x_{\lambda_1 + \lambda_2 + 1}, y_t)}, 
 (E_{3}^{k} \otimes F_4) \otimes w_{(x_{\lambda_1 + \lambda_2 + 1}, y_t)} \} + Z,
\end{align*}
where $E_{3}^{-1} = E_{3}^{-2} = 0$.
\end{prop}

\begin{proof}
The proof follows from \eqref{j1}, \eqref{j2} and from
\begin{align*}
& (F_1 \otimes E_3) \otimes w_{(x_{\lambda_1 + \lambda_2 + 1}, y_t)} \in (1 \otimes F_4) \otimes w_{(x_{\lambda_1 + \lambda_2 + 1}, y_t)} + Z \\
& (F_4 \otimes E_3) \otimes w_{(x_{\lambda_1 + \lambda_2 + 1}, y_t)} = (E_3 \otimes F_4) \otimes w_{(x_{\lambda_1 + \lambda_2 + 1}, y_t)}.
\end{align*}
\end{proof}
Finally, from propositions \ref{l1} and \ref{l2} and previous conclusions we get
\begin{theorem}\label{si}
One spanning set of the vector space $\mathcal{A} \otimes_{\mathcal{B}} W$ is given by
\begin{align*}
\{ & (F_{2}^{c} E_{3}^{k} \otimes E_3) \otimes w_{(x_s, y_t)}, (F_{2}^{c} E_{3}^{k} \otimes E_4) \otimes w_{(x_s, y_t)}, \\
& (F_{2}^{c} E_{3}^{k} \otimes 1) \otimes w_{(x_s, y_t)}, (F_{2}^{c} E_{3}^{k} \otimes E_3 E_4) \otimes w_{(x_s, y_t)}, \\
& (x \otimes F_3) \otimes w_{(x_{\lambda_1 + \lambda_2 + 1}, y_t)}, (x \otimes F_4) \otimes w_{(x_{\lambda_1 + \lambda_2 + 1}, y_t)}, \\
& (x \otimes E_3 F_3) \otimes w_{(x_{\lambda_1 + \lambda_2 + 1}, y_t)}, (x \otimes F_3 F_4) \otimes w_{(x_{\lambda_1 + \lambda_2 + 1}, y_t)}, \\
& x \in \{ F_{1}^{a} F_{2}^{b} E_{3}^{k} \, | \, k \in \mathbb{N}_{0}, a, b \in \{ 0, 1, \cdots, k \} \},  c \in \{ 0, 1, \cdots, k \}, \\
& s \in \{ 0, 1, \cdots, \lambda_1 + \lambda_2 + 1 \}, t \in \{ 0, 1, \cdots, \lambda_1 - \lambda_2 \} \}.
\end{align*}
\end{theorem}
\newpage
We will show that the above spanning set is also a basis of $\mathcal{A} \otimes_{\mathcal{B}}W$ a bit later. In the figure below $ c \in \{ 0, 1, \cdots, k \}$ and $t \in \{ 0, 1, \cdots, \lambda_1 - \lambda_2 \}$.
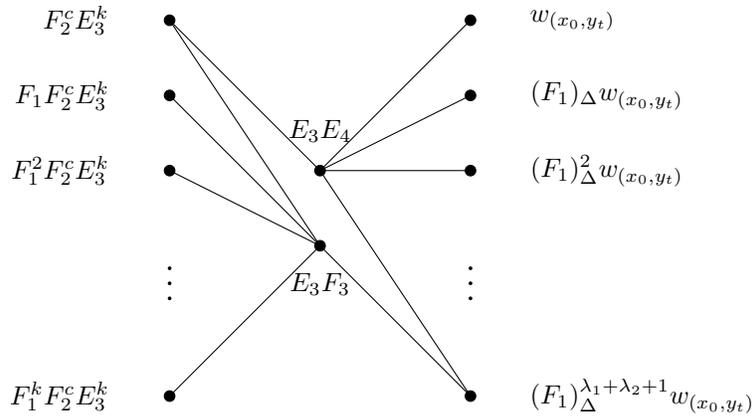
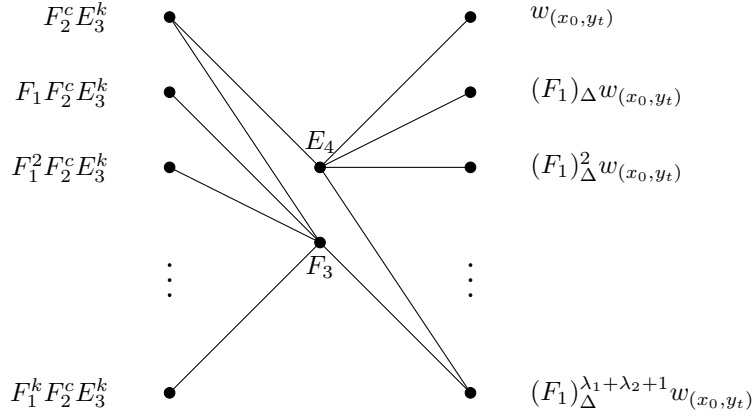
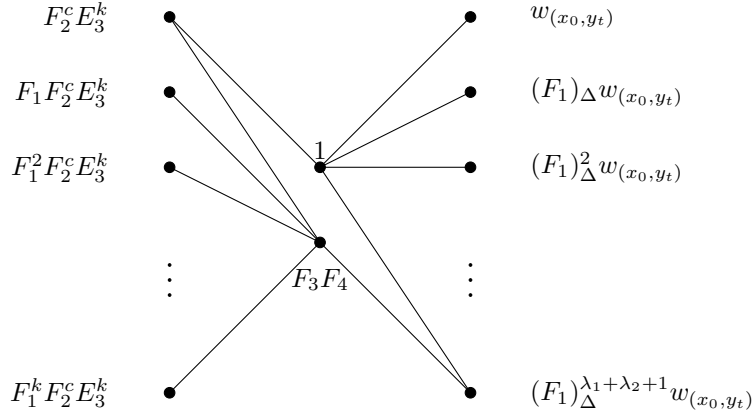
\begin{figure}[ht]
\centering
\caption{Basis}
\bigskip
\bigskip
\begin{tikzpicture}
\tikzstyle{every node}=[draw = black,shape=circle, fill, inner sep=0pt,minimum size=4pt]
\node[label=left: $F_{2}^{c} E_{3}^{k} \qquad$] at (0,0) {};
\node[label=left: $F_1 F_{2}^{c} E_{3}^{k} \qquad$] at (0,-1) {};
\node[label=left: $F_{1}^{2} F_{2}^{c} E_{3}^{k} \qquad$] at (0,-2) {};

\filldraw [black] (0, -3.3) circle (0.5pt);
\filldraw [black] (0, -3.5) circle (0.5pt);
\filldraw [black] (0, -3.7) circle (0.5pt);

\node[label=left: $F_{1}^{k} F_{2}^{c} E_{3}^{k} \qquad$] at (0,-5) {};

\node[label=above: $E_3$] at (2,-2) {};
\node[label=below: $F_4$] at (2,-3) {};

\node[label=right: $ \qquad w_{(x_0, y_t)}$] at (4,0) {};
\node[label=right: $ \qquad (F_{1})_{\Delta} w_{(x_0, y_t)}$] at (4,-1) {};
\node[label=right: $\qquad (F_{1})_{\Delta}^{2} w_{(x_0, y_t)}$] at (4,-2) {};

\filldraw [black] (4, -3.3) circle (0.5pt);
\filldraw [black] (4, -3.5) circle (0.5pt);
\filldraw [black] (4, -3.7) circle (0.5pt);

\node[label=right: $\qquad (F_{1})_{\Delta}^{\lambda_1 + \lambda_2 + 1} w_{(x_0, y_t)}$] at (4,-5) {};

\draw (0, 0) -- (2, -2);
\draw (4, 0) -- (2, -2);
\draw (4, -1) -- (2, -2);
\draw (4, -2) -- (2, -2);
\draw (4, -5) -- (2, -2);

\draw (0, 0) -- (2, -3);
\draw (0, -1) -- (2, -3);
\draw (0, -2) -- (2, -3);
\draw (0, -5) -- (2, -3);
\draw (4, -5) -- (2, -3);

\end{tikzpicture}
\end{figure}
\begin{figure}[ht]
\centering
\begin{tikzpicture}
\tikzstyle{every node}=[draw = black,shape=circle, fill, inner sep=0pt,minimum size=4pt]
\node[label=left: $F_{2}^{c} E_{3}^{k} \qquad$] at (0,0) {};
\node[label=left: $F_1 F_{2}^{c} E_{3}^{k} \qquad$] at (0,-1) {};
\node[label=left: $F_{1}^{2} F_{2}^{c} E_{3}^{k} \qquad$] at (0,-2) {};

\filldraw [black] (0, -3.3) circle (0.5pt);
\filldraw [black] (0, -3.5) circle (0.5pt);
\filldraw [black] (0, -3.7) circle (0.5pt);

\node[label=left: $F_{1}^{k} F_{2}^{c} E_{3}^{k} \qquad$] at (0,-5) {};

\node[label=above: $E_3 E_4$] at (2,-2) {};
\node[label=below: $E_3 F_3$] at (2,-3) {};

\node[label=right: $ \qquad w_{(x_0, y_t)}$] at (4,0) {};
\node[label=right: $ \qquad (F_{1})_{\Delta} w_{(x_0, y_t)}$] at (4,-1) {};
\node[label=right: $\qquad (F_{1})_{\Delta}^{2} w_{(x_0, y_t)}$] at (4,-2) {};

\filldraw [black] (4, -3.3) circle (0.5pt);
\filldraw [black] (4, -3.5) circle (0.5pt);
\filldraw [black] (4, -3.7) circle (0.5pt);

\node[label=right: $\qquad (F_{1})_{\Delta}^{\lambda_1 + \lambda_2 + 1} w_{(x_0, y_t)}$] at (4,-5) {};

\draw (0, 0) -- (2, -2);
\draw (4, 0) -- (2, -2);
\draw (4, -1) -- (2, -2);
\draw (4, -2) -- (2, -2);
\draw (4, -5) -- (2, -2);

\draw (0, 0) -- (2, -3);
\draw (0, -1) -- (2, -3);
\draw (0, -2) -- (2, -3);
\draw (0, -5) -- (2, -3);
\draw (4, -5) -- (2, -3);

\end{tikzpicture}
\end{figure}
\begin{figure}[ht]
\centering
\begin{tikzpicture}
\tikzstyle{every node}=[draw = black,shape=circle, fill, inner sep=0pt,minimum size=4pt]
\node[label=left: $F_{2}^{c} E_{3}^{k} \qquad$] at (0,0) {};
\node[label=left: $F_1 F_{2}^{c} E_{3}^{k} \qquad$] at (0,-1) {};
\node[label=left: $F_{1}^{2} F_{2}^{c} E_{3}^{k} \qquad$] at (0,-2) {};

\filldraw [black] (0, -3.3) circle (0.5pt);
\filldraw [black] (0, -3.5) circle (0.5pt);
\filldraw [black] (0, -3.7) circle (0.5pt);

\node[label=left: $F_{1}^{k} F_{2}^{c} E_{3}^{k} \qquad$] at (0,-5) {};

\node[label=above: $E_4$] at (2,-2) {};
\node[label=below: $F_3$] at (2,-3) {};

\node[label=right: $ \qquad w_{(x_0, y_t)}$] at (4,0) {};
\node[label=right: $ \qquad (F_{1})_{\Delta} w_{(x_0, y_t)}$] at (4,-1) {};
\node[label=right: $\qquad (F_{1})_{\Delta}^{2} w_{(x_0, y_t)}$] at (4,-2) {};

\filldraw [black] (4, -3.3) circle (0.5pt);
\filldraw [black] (4, -3.5) circle (0.5pt);
\filldraw [black] (4, -3.7) circle (0.5pt);

\node[label=right: $\qquad (F_{1})_{\Delta}^{\lambda_1 + \lambda_2 + 1} w_{(x_0, y_t)}$] at (4,-5) {};

\draw (0, 0) -- (2, -2);
\draw (4, 0) -- (2, -2);
\draw (4, -1) -- (2, -2);
\draw (4, -2) -- (2, -2);
\draw (4, -5) -- (2, -2);

\draw (0, 0) -- (2, -3);
\draw (0, -1) -- (2, -3);
\draw (0, -2) -- (2, -3);
\draw (0, -5) -- (2, -3);
\draw (4, -5) -- (2, -3);

\end{tikzpicture}
\end{figure}
\begin{figure}[ht]
\centering
\begin{tikzpicture}
\tikzstyle{every node}=[draw = black,shape=circle, fill, inner sep=0pt,minimum size=4pt]
\node[label=left: $F_{2}^{c} E_{3}^{k} \qquad$] at (0,0) {};
\node[label=left: $F_1 F_{2}^{c} E_{3}^{k} \qquad$] at (0,-1) {};
\node[label=left: $F_{1}^{2} F_{2}^{c} E_{3}^{k} \qquad$] at (0,-2) {};

\filldraw [black] (0, -3.3) circle (0.5pt);
\filldraw [black] (0, -3.5) circle (0.5pt);
\filldraw [black] (0, -3.7) circle (0.5pt);

\node[label=left: $F_{1}^{k} F_{2}^{c} E_{3}^{k} \qquad$] at (0,-5) {};

\node[label=above: $1$] at (2,-2) {};
\node[label=below: $F_3 F_4$] at (2,-3) {};

\node[label=right: $ \qquad w_{(x_0, y_t)}$] at (4,0) {};
\node[label=right: $ \qquad (F_{1})_{\Delta} w_{(x_0, y_t)}$] at (4,-1) {};
\node[label=right: $\qquad (F_{1})_{\Delta}^{2} w_{(x_0, y_t)}$] at (4,-2) {};

\filldraw [black] (4, -3.3) circle (0.5pt);
\filldraw [black] (4, -3.5) circle (0.5pt);
\filldraw [black] (4, -3.7) circle (0.5pt);

\node[label=right: $\qquad (F_{1})_{\Delta}^{\lambda_1 + \lambda_2 + 1} w_{(x_0, y_t)}$] at (4,-5) {};

\draw (0, 0) -- (2, -2);
\draw (4, 0) -- (2, -2);
\draw (4, -1) -- (2, -2);
\draw (4, -2) -- (2, -2);
\draw (4, -5) -- (2, -2);

\draw (0, 0) -- (2, -3);
\draw (0, -1) -- (2, -3);
\draw (0, -2) -- (2, -3);
\draw (0, -5) -- (2, -3);
\draw (4, -5) -- (2, -3);

\end{tikzpicture}
\end{figure}

\section{Construction of the discrete series representations of the group $SO_{e}(4,1)$ via Dirac induction}
Let $X = A_{\mathfrak{b}}(\lambda)$ be a discrete series as in Section $2$. Our goal is to get a basis of $X \otimes S$ which is similar to the spanning set from the theorem \ref{si}. The first step is to ``remove" $E_4$ from Theorem \ref{basis1}. Since
$$
(\text{ad}F_2) E_{3}^{k} = k E_{3}^{k - 1} E_{4} - k(k - 1) E_{3}^{k - 2} E_{1},
$$
one can easily show by induction that for all $k \in \mathbb{N}_{0}$ and for all $l \in \{ 0, 1, \cdots, \lambda_1 - \lambda_2 \}$ we have
$$
(\text{ad}F_2)^{l} E_{3}^{k + l} \cdot v_{\lambda_1 + \lambda_2 + 2} \boxtimes v_{\lambda_1 - \lambda_2} = (k + l)(k + l - 1) \cdots (k + 1)E_{3}^{k} E_{4}^{l} \cdot v_{\lambda_1 + \lambda_2 + 2} \boxtimes v_{\lambda_1 - \lambda_2}.
$$
One basis of the vector space $Z_t$ (see \eqref{Zt1}) is given by 
\begin{align}\label{Zt2}
S_{t}^{1} = \{ F_{1}^{a} F_{2}^{b} (\text{ad}F_2)^{l} E_{3}^{k + l} \cdot v_{\lambda_1 + \lambda_2  + 2} \boxtimes v_{\lambda_1 - \lambda_2} \, | \, & k \in \mathbb{N}_{0}, l \in \{ 0, 1, \cdots, \lambda_1 - \lambda_2 \}, k + l \leq t,  \notag \\
& a \in \{ 0, 1, \cdots , \lambda_1 + \lambda_2  + 2 + k + l \}, \notag \\
& b \in \{ 0, 1, \cdots , \lambda_1 - \lambda_2 + k - l \} \}.
\end{align}
Furthermore, we have 
\begin{equation}\label{nizi}
(\text{ad}F_2)^{l} E_{3}^{k + l} \in \Span \{ F_{2}^{s} E_{3}^{k + l} F_{2}^{t'} \, | \, s + t' = l \}.
\end{equation}
From (\ref{Zt2}) and \eqref{nizi} it easily follows that the set 
\begin{align}\label{Zt3}
S_{t}^{2} = \{ F_{1}^{a} F_{2}^{b} E_{3}^{k + l} F_{2}^{l} \cdot v_{\lambda_1 + \lambda_2  + 2} \boxtimes v_{\lambda_1 - \lambda_2} \, | \, & k \in \mathbb{N}_{0}, l \in \{ 0, 1, \cdots, \lambda_1 - \lambda_2 \}, k + l \leq t, \notag \\
& a \in \{ 0, 1, \cdots , \lambda_1 + \lambda_2  + 2 + k + l \}, \notag \\
& b \in \{ 0, 1, \cdots , \lambda_1 - \lambda_2 + k - l \} \}
\end{align}
is a spanning set of the vector space $Z_t$. Since the sets $S_{t}^{1}$ and $S_{t}^{2}$ have the same cardinality, the set $S_{t}^{2}$ is a basis of $Z_t$. If we denote $r = k + l$ it follows that one basis of the vector space $Z_t \otimes E_3$ is given by
\begin{align}\label{ZtE3}
(S_{t})_{E_3}^{1} = \{ F_{1}^{a} F_{2}^{b} E_{3}^{r} \otimes E_3 \cdot w_{(\lambda_1 + \frac{1}{2}, y_l)} \, | \, & r \in \{0, \cdots, t \}, l \in \{ 0, 1, \cdots, \text{min}\{\lambda_1 - \lambda_2, r \} \},  \notag \\
& a \in \{ 0, 1, \cdots , \lambda_1 + \lambda_2  + 2 + r \}, \notag \\
& b \in \{ 0, 1, \cdots , \lambda_1 - \lambda_2 + r - 2l \} \}.
\end{align}
To prove what we claim we need the following lemma.
\begin{lemma}\label{smanjiF2}
For $x, y \in \mathbb{N}_{0}$ such that $x > y$ and for $l \in \{ 0, 1, \cdots, \lambda_1 - \lambda_2 \} $ we have
\begin{align*}
F_{2}^{x} E_{3}^{y} \otimes E_3 \cdot w_{(\lambda_1 + \frac{1}{2}, y_l)} \in \Span \{ & F_{2}^{z} E_{3}^{v} \otimes E_3 \cdot w_{(\lambda_1 + \frac{1}{2}, \lambda_2 + \frac{1}{2} - p)} \, | \, \\
& p \in \{ 0, 1, \cdots, \lambda_1 - \lambda_2 \}, v \in \{ 0, 1, \cdots, y \}, z \leq v \}.
\end{align*}
\end{lemma}
\begin{proof}
It is sufficient to prove the claim for $x = y + 1$. The general case then follows easily by induction. The claim is obvious for $y = 0$. Let us assume that the claim is true for some $y = k - 1$. Then we have
\begin{align}\label{k+1}
F_{2}^{k + 1} E_{3}^{k} \otimes E_3 \cdot w_{(\lambda_1 + \frac{1}{2}, y_l)} = & k E_{4} F_{2}^{k} E_{3}^{k - 1} \otimes E_3 \cdot w_{(\lambda_1 + \frac{1}{2}, y_l)} \notag \\
&  +  F_{2}^{k} E_{3}^{k} \otimes E_3 \cdot w_{(\lambda_1 + \frac{1}{2}, y_l - 1)},
\end{align}
where $w_{(\lambda_1 + \frac{1}{2}, y_l - 1)} = 0$ for $l = \lambda_1 - \lambda_2$. Furthermore, we have
\begin{align}\label{E4(k + 1)}
E_4 F_{2}^{z} E_{3}^{v} \otimes E_3 \cdot w_{(\lambda_1 + \frac{1}{2}, \lambda_2 + \frac{1}{2} - p)} & = F_{2}^{z} E_{3}^{v} E_4 \otimes E_3 \cdot w_{(\lambda_1 + \frac{1}{2}, \lambda_2 + \frac{1}{2} - p)} \notag \\
& \in \Span \{ F_{2}^{z} (\text{ad} F_2) E_{3}^{v + 1} \otimes E_3 \cdot w_{(\lambda_1 + \frac{1}{2}, \lambda_2 + \frac{1}{2} - p)} \} \notag \\
& \in \Span \{ F_{2}^{z + 1} E_{3}^{v + 1} \otimes E_3 \cdot w_{(\lambda_1 + \frac{1}{2}, \lambda_2 + \frac{1}{2} - p)} \notag \\
& + F_{2}^{z} E_{3}^{v + 1} \otimes E_3 \cdot w_{(\lambda_1 + \frac{1}{2}, \lambda_2 + \frac{1}{2} - p - 1)} \}.
\end{align}
If $z \leq v \leq k - 1$ then $z + 1 \leq v + 1 \leq k$. The proof follows from (\ref{k+1}) and (\ref{E4(k + 1)}).
\end{proof}
Let us consider the following set
\begin{align}\label{ZtE31}
(S_{t})_{E_3}^{2} = \{ F_{1}^{a} F_{2}^{b} E_{3}^{r} \otimes E_3 \cdot w_{(\lambda_1 + \frac{1}{2}, y_l)} \, | \, & r \in \{0, \cdots, t \}, l \in \{ 0, 1, \cdots, \lambda_1 - \lambda_2 \},  \notag \\
& a \in \{ 0, 1, \cdots , \lambda_1 + \lambda_2  + 2 + r \}, \notag \\
& b \in \{ 0, 1, \cdots , r \} \}.
\end{align}
It follows from lemma \ref{smanjiF2} that the set $(S_{t})_{E_3}^{2} \subset Z_t \otimes E_3$ is a spanning set of $Z_t \otimes E_3$. Since the cardinality of the set $(S_{t})_{E_3}^{2}$ is 
$$
\sum_{r = 0}^{t}(m_0 + r + 1)(r + 1)(n_0 + 1),
$$
(where $m_0 = \lambda_1 + \lambda_2 + 2, n_0 = \lambda_1 - \lambda_2$) and the dimension of the space
$$
Z_t \otimes E_3 =  (V_{m + r} \boxtimes (V_{n + r} \oplus V_{n + r - 2} \oplus \cdots \oplus V_{|n - r|})) \otimes E_3
$$
is 
$$
\sum_{r = 0}^{t}(m_0 + r + 1)(\text{dim}V_{n_0 + r} + \text{dim}V_{n_0 + r - 2} + \cdots + \text{dim}V_{|n_0 - r|}) = \sum_{r = 0}^{t}(m_0 + r + 1)(r + 1)(n_0 + 1),
$$
the set $(S_{t})_{E_3}^{2}$ is a basis of $Z_t \otimes E_3$.
\vspace{.2in}

The next proposition is an analogue to the proposition \ref{l1} in the space $X \otimes S = A_{\mathfrak{b}}(\lambda) \otimes S$.
\begin{prop}\label{m1}
For $s \in \{0, 1, \cdots, \lambda_1 + \lambda_2 \}$, $l \in \{0, 1, \cdots, \lambda_1 - \lambda_2 \}$ and for any nonnegative integer $k$ we have
\begin{align*}
(F_1 E_{3}^{k} \otimes E_3) \cdot w_{(x_s, y_l)} \in \Span_{\mathbb{C}}\{ (E_{3}^{k - 2} \otimes E_3) \cdot w_{(x_s, y_{l - 1})}, \\
(E_{3}^{k} \otimes E_3) \cdot w_{(x_{s + 1}, y_l)} \},
\end{align*}
where $E_{3}^{-1} = E_{3}^{-2} = 0$.
\end{prop}

\proof
For each $k \in \mathbb{N}$ 
\begin{equation}\label{indukcija}
F_1 E_{3}^{k} = -k(k-1) E_{3}^{k - 2} E_2 + k E_{3}^{k - 1} F_4 + E_{3}^{k} F_1.
\end{equation}
Furthermore, we have
\begin{equation}\label{k2} 
(E_2 \otimes E_3) \cdot w_{(x_s, y_l)}  = l (\lambda_1 - \lambda_2 - l + 1) (1 \otimes E_3) \cdot w_{(x_s, y_{l - 1})}, 
\end{equation}
and
\begin{equation}\label{k3}
(F_1 \otimes E_3) \cdot w_{(x_s, y_l)} = \frac{\lambda_1 + \lambda_2 + 2 - s}{\lambda_1 + \lambda_2 + 1 - s}(1 \otimes E_3) \cdot w_{(x_{s + 1}, y_l)}.
\end{equation}
Since
$(E_3 \otimes F_4 - F_4 \otimes E_3) \cdot w = 0$ for all $w \in W$, 
we get
\begin{align}\label{k4}
& (F_4 \otimes E_3) \cdot w_{(x_s, y_l)} = (E_3 \otimes F_4) \cdot w_{(x_s, y_l)} \\
& = \frac{1}{\lambda_1 + \lambda_2 + 1 - s}(E_3 \otimes E_3) \cdot w_{(x_{s + 1}, y_l)}. \notag 
\end{align}
The proof follows from \eqref{indukcija}, \eqref{k2}, \eqref{k3} and (\ref{k4}).
\qed

\bigskip
The next proposition is an analogue to the proposition \ref{l2} in the space $X \otimes S = A_{\mathfrak{b}}(\lambda) \otimes S$.
\begin{prop}\label{m2}
For $l \in \{0, 1, \cdots, \lambda_1 - \lambda_2 \}$ and for any nonnegative integer $k$ we have
\begin{align*}
& (F_1 E_{3}^{k} \otimes E_3) \cdot w_{(x_{\lambda_1 + \lambda_2 + 1}, y_l)} \\
& \in \Span_{\mathbb{C}}\{ (E_{3}^{k - 2} \otimes E_3) \cdot w_{(x_{\lambda_1 + \lambda_2 + 1}, y_{l - 1})}, 
(E_{3}^{k} \otimes F_4) \cdot w_{(x_{\lambda_1 + \lambda_2 + 1}, y_l)} \},
\end{align*}
where $E_{3}^{-1} = E_{3}^{-2} = 0$.
\end{prop}

\begin{proof}
The proof follows from \eqref{j1}, \eqref{k2} and
\begin{align*}
& (F_1 \otimes E_3) \cdot w_{(x_{\lambda_1 + \lambda_2 + 1}, y_l)} = (1 \otimes F_4) \cdot w_{(x_{\lambda_1 + \lambda_2 + 1}, y_l)} ; \\
& (F_4 \otimes E_3) \cdot w_{(x_{\lambda_1 + \lambda_2 + 1}, y_l)} = (E_3 \otimes F_4) \cdot w_{(x_{\lambda_1 + \lambda_2 + 1}, y_l)}.
\end{align*}
\end{proof}
From propositions \ref{m1} and \ref{m2} we get that the set
\begin{align}\label{ZtE32}
(S_t)_{\text{final}} = \{ & (F_{2}^{c} E_{3}^{k} \otimes E_3) \cdot w_{(x_s, y_l)},  \\
& (F_{1}^{a} F_{2}^{b} E_{3}^{k} \otimes F_4) \otimes w_{(x_{\lambda_1 + \lambda_2 + 1}, y_l)}, \, | \,  \notag\\
& k \in \{ 0, 1, \cdots, t \}, a, b, c \in \{ 0, 1, \cdots, k \}, \notag \\
& s \in \{ 0, 1, \cdots, \lambda_1 + \lambda_2 + 1 \}, l \in \{ 0, 1, \cdots, \lambda_1 - \lambda_2 \} \} \subset Z_t \otimes E_3 \notag
\end{align}
is a spanning set of $Z_t \otimes E_3$. Since the cardinality of the set $(S_t)_{\text{final}}$ equals the cardinality of the set $(S_t)_{E_3}^{2}$ which is a basis of $Z_t \otimes E_3$, we conclude that $(S_t)_{\text{final}}$ is a basis of $Z_t \otimes E_3$.
Finally, from the previous conclusions we get that one basis of the vector space $X \otimes S$ is given by
\begin{align*}
\{ & (F_{2}^{c} E_{3}^{k} \otimes E_3) \cdot w_{(x_s, y_l)}, (F_{2}^{c} E_{3}^{k} \otimes E_4) \cdot w_{(x_s, y_l)}, \\
& (F_{2}^{c} E_{3}^{k} \otimes 1) \cdot w_{(x_s, y_l)}, (F_{2}^{c} E_{3}^{k} \otimes E_3 E_4) \cdot w_{(x_s, y_l)}, \\
& (x \otimes F_3) \cdot w_{(x_{\lambda_1 + \lambda_2 + 1}, y_l)}, (x \otimes F_4) \cdot w_{(x_{\lambda_1 + \lambda_2 + 1}, y_l)}, \\
& (x \otimes E_3 F_3) \cdot w_{(x_{\lambda_1 + \lambda_2 + 1}, y_l)}, (x \otimes F_3 F_4) \cdot w_{(x_{\lambda_1 + \lambda_2 + 1}, y_l)}, \\
& x \in \{ F_{1}^{a} F_{2}^{b} E_{3}^{k} \, | \, k \in \mathbb{N}_{0}, a, b \in \{ 0, 1, \cdots, k \} \},  c \in \{ 0, 1, \cdots, k \}, \\
& s \in \{ 0, 1, \cdots, \lambda_1 + \lambda_2 + 1 \}, l \in \{ 0, 1, \cdots, \lambda_1 - \lambda_2 \} \}.
\end{align*}
The action map $\phi : \mathcal{A} \otimes W \longrightarrow X \otimes S$
$$
\phi(a \otimes w) = a \cdot w, \quad a \in \mathcal{A}, w \in W
$$
maps the spanning set of $\mathcal{A} \otimes_{\mathcal{B}} W$ from Theorem \ref{si} into the basis of the space $X \otimes S$, which implies that this spanning set is a basis of the vector space $\mathcal{A} \otimes_{\mathcal{B}} W$. Finally, we get:
\begin{theorem}\label{themain}
Let $X$ be the vector space of the discrete series representation of the group $SO_{e}(4,1)$ and let $W$ be the Dirac cohomology of $X$. Then the
$(\mathcal{A}, \tilde{K})$--module $\text{Ind}_{D}(W) = \mathcal{A} \otimes_{\mathcal{B}} W$ is isomorphic to the $(\mathcal{A}, \tilde{K})$--module $X \otimes S$.
\end{theorem}

\begin{remark}
There is one more positive root system that contains the fixed positive root system $\{\epsilon_1 \pm \epsilon_2\}$ for $\ka$. This positive root system is
$$
\{\epsilon_1 \pm \epsilon_2, \epsilon_1, - \epsilon_2\}
$$
and the corresponding Borel subalgebra is 
$$
\mathfrak{b'} = \h \oplus \g_{\epsilon_1 - \epsilon_2} \oplus \g_{\epsilon_1 + \epsilon_2} \oplus \g_{\epsilon_1} \oplus \g_{-\epsilon_2}.
$$
Discrete series representations of $SO_{e}(4,1)$ of the form $A_{\mathfrak{b'}}(\lambda)$  for some $\lambda$ which is admissible for $\mathfrak{b'}$ can also be constructed via Dirac induction. The proof is similar. 
\end{remark}

\begin{remark}
In \cite{PR} the holomorphic discrete series representations were constructed via intermediate version of induction where we tensor over the algebra $\mathcal{B}_1$ generated by $U(\ka_{\Delta})$, by the ideal $\mathcal{I}$ and $C(\p)^{K}$. Here and in \cite{Pr1} we used the reduced version of induction where we tensor over the algebra $\mathcal{B} = U(\ka_{\Delta})\mathcal{A}^{K}$. The remaining question is whether we really need the whole algebra $\mathcal{B}$ or is it possible to construct nonholomorphic discrete series of $SO_{e}(4,1)$ by tensoring over the subalgebra $\mathcal{B}_1 \subset \mathcal{B}$ or even the smaller subalgebra $\mathcal{B}_2$ generated by $U(\ka_{\Delta})$ and the ideal $\mathcal{I}$. 

From  \cite[Corollary~3.2.]{PR}, it follows that an irreducible $\tilde{K}$--module $W$ is contained in the Dirac cohomology of an irreducible $(\g, K)$--module $Y$ if and only if $Y \otimes S$ is a quotient of $Ind_D(W)$. Let $X = A_{\mathfrak{b}}(\lambda)$ be the discrete series representation of $G$ for the Borel subalgebra
$$
\mathfrak{b} = \h \oplus \g_{\epsilon_1 - \epsilon_2} \oplus \g_{\epsilon_1 + \epsilon_2} \oplus \g_{\epsilon_1} \oplus \g_{\epsilon_2}.
$$
and some $\lambda = (\lambda_1, \lambda_2) \in \h^{*}$ which is admissible for $\mathfrak{b}$. Let $F_{\lambda}$ be the irreducible finite-dimensional $(\g, K)$--module with highest weight $\lambda$. Then, from \cite[Theorem~4.2.]{HKP} it follows that
$$
H^{D}(F_{\lambda}) = V_{(\lambda_1 + \frac{1}{2}, \lambda_2 + \frac{1}{2})} \oplus V_{(\lambda_1 + \frac{1}{2}, -\lambda_2 - \frac{1}{2})},
$$
where $V_{(\lambda_1 + \frac{1}{2}, \lambda_2 + \frac{1}{2})}$ is the Cartan component of the tensor product $V_{(\lambda_1, \lambda_2)} \otimes V_{(\frac{1}{2}, \frac{1}{2})}$ for the unique $K$--type $V_{(\lambda_1, \lambda_2)}$ of $F_{\lambda}$, and $V_{(\lambda_1 + \frac{1}{2}, -\lambda_2 - \frac{1}{2})}$ is the Cartan component of the tensor product $V_{(\lambda_1, -\lambda_2)} \otimes V_{(\frac{1}{2}, -\frac{1}{2})}$ for the unique $K$--type $V_{(\lambda_1, -\lambda_2)}$ of $F_{\lambda}$. We note that \cite[Theorem~4.2.]{HKP} is stated for Vogan's version of Dirac cohomology, but it also holds for Dirac cohomology we are using here because of \cite[Proposition~2.5]{PR}, which states that $$H_D(X) = H^{D}(X) = H_{V}^{D}(X) = \Ker D = \Ker D^2$$ if the module $X$ is either unitary or finite-dimensional.

Since the Dirac cohomology $W$ of the discrete series representation $X$ is the PRV component of the tensor product $V_{(\lambda_1 + 1, \lambda_2 + 1)} \otimes V_{(\frac{1}{2}, \frac{1}{2})}$, where $V_{(\lambda_1 + 1, \lambda_2 + 1)}$ is the lowest $K$--type of $X$, we see that $V_{(\lambda_1 + \frac{1}{2}, \lambda_2 + \frac{1}{2})} \subset H^{D}(F_{\lambda})$ is isomorphic to $H^{D}(X)$ not only as $(\mathcal{B}_2, \tilde{K})$ but also as $(\mathcal{B}_1, \tilde{K})$--modules, since $C(\p)^K$ acts only on $V_{(\frac{1}{2}, \frac{1}{2})}$ and not on $V_{(\lambda_1 + 1, \lambda_2 + 1)} \subset A_{\mathfrak{b}}(\lambda)$ or $V_{(\lambda_1, \lambda_2)} \subset F_{\lambda}$. Therefore, $F_{\lambda} \otimes S$ is a quotient of $Ind_D(W)$ in case we tensor over the algebra $\mathcal{B}_1$ or $\mathcal{B}_2$. Of course, from Theorem \ref{themain} it follows that $F_{\lambda} \otimes S$ is not a quotient of $Ind_D(W)$ if we tensor over the algebra $\mathcal{B}$. This can also be proved in the following way.
Since the Casimir element $\Omega_{\ka}$ of $\ka$  acts by the scalar 
$$
(\lambda_1 + \lambda_2 + 2) + (\lambda_1 - \lambda_2) + \frac{1}{2}(\lambda_1 + \lambda_2 + 2)^2 + \frac{1}{2}(\lambda_1 - \lambda_2)^2
$$
on the PRV component $W = V_{(\lambda_1 + \frac{1}{2}, \lambda_2 + \frac{1}{2})} \subset V_{(\lambda_1 + 1, \lambda_2 + 1)} \otimes V_{(\frac{1}{2}, \frac{1}{2})} $, and by the scalar
$$
(\lambda_1 + \lambda_2) + (\lambda_1 - \lambda_2) + \frac{1}{2}(\lambda_1 + \lambda_2)^2 + \frac{1}{2}(\lambda_1 - \lambda_2)^2
$$
on the Cartan component $V_{(\lambda_1 + \frac{1}{2}, \lambda_2 + \frac{1}{2})} \subset V_{(\lambda_1, \lambda_2)} \otimes V_{(\frac{1}{2}, \frac{1}{2})}$, from the identity
\begin{align*}
\Omega_{\ka_{\Delta}} & = \Omega_{\ka} + D_{\ka} + \frac{1}{2} \cdot 1 \otimes \alpha(H_1 + H_2)^2 + \frac{1}{2} \cdot 1 \otimes \alpha(H_1 - H_2)^2 \\
& + 1 \otimes (\alpha(E_1) \alpha(F_1) + \alpha(F_1) \alpha(E_1) + \alpha(E_2) \alpha(F_2) + \alpha(F_2) \alpha(E_2))
\end{align*}
it follows that $D_{\ka} \in \mathcal{B} \setminus \mathcal{B}_1$ acts by different scalars on $W$ and on $V_{(\lambda_1 + \frac{1}{2}, \lambda_2 + \frac{1}{2})} \subset H^{D}(F_{\lambda})$. Therefore,  $F_{\lambda} \otimes S$ is not a quotient of $Ind_D(W)$ if we tensor over the algebra $\mathcal{B}$. 
\end{remark}

\textbf{Acknowledgements:} The author is grateful to Pavle Pand\v{z}i\'{c} for a careful reading of the paper and his comments.

\end{document}